\documentclass[11pt]{article}

\usepackage[margin=1.3in]{geometry}
\usepackage{amsmath, amsfonts, amssymb, amsthm}
\usepackage{graphicx}
\usepackage{hyperref}
\hypersetup{
    colorlinks,%
    citecolor=black,%
    filecolor=black,%
    linkcolor=black,%
    urlcolor=black
}
\usepackage{color}
\usepackage{framed}

\newtheorem{thm}{Theorem}[section]
\newtheorem{lem}[thm]{Lemma}
\newtheorem{prop}[thm]{Proposition}
\newtheorem{defn}[thm]{Definition}

\newtheorem{cor}[thm]{Corollary}

\numberwithin{equation}{section}

\def\XXint#1#2#3{{\setbox0=\hbox{$#1{#2#3}{\int}$}
  \vcenter{\hbox{$#2#3$}}\kern-.5\wd0}}

\newcommand{\al}{\alpha}                
                
\newcommand{\va}{\varepsilon}           \newcommand{\eps}{\varepsilon}           \newcommand{\ud}{\mathrm{d}}
\newcommand{\be}{\begin{equation}}      \newcommand{\ee}{\end{equation}}

\newcommand{\R}{\mathbb{R}}

\newcommand{\Div}{\mbox{ div}}
\newcommand{\osc}{\mbox{ osc}}

\begin{document}

\date{\today}
\title{\textbf{H\"older gradient estimates for  parabolic homogeneous $p$-Laplacian equations}}

\author{\medskip Tianling Jin\footnote{Supported in part by NSF grant DMS-1362525.} \quad and \quad Luis Silvestre\footnote{Supported in part by NSF grants DMS-1254332 and DMS-1065979.}}
\maketitle

\begin{abstract}
We prove interior H\"older  estimates for the spatial gradient of viscosity solutions to the parabolic homogeneous $p$-Laplacian equation
\[
 u_t=|\nabla u|^{2-p} \mbox{ div} (|\nabla u|^{p-2}\nabla u),
\]
where $1<p<\infty$. This equation arises from \emph{tug-of-war}-like stochastic games with noise. It can also be considered as the parabolic $p$-Laplacian equation in non divergence form.
\end{abstract}

\section{Introduction}
For $1<p<\infty$, the $p$-Laplacian equation
\begin{equation}\label{eq:p div}
 \Delta_p u:= \Div (|\nabla u|^{p-2}\nabla u)=0
\end{equation}
is the Euler-Lagrange equation of the energy functional
\begin{equation}\label{eq:p functional}
 \frac 1p \int |\nabla u(x)|^p d x,
\end{equation}
where $\nabla$ and $\Div$ are the gradient and divergence operators in the variable $x\in\R^n$. It is a classical result that every weak solution of \eqref{eq:p div} in the distribution sense is $C^{1,\alpha}$ for some $\alpha>0$. This result and its various proofs can be found in, e.g., Ural'ceva \cite{Ural}, Uhlenbeck \cite{Uhlenbeck}, Evans \cite{Evans2}, DiBenedetto \cite{DiBenedetto2}, Lewis \cite{Lewis}, Tolksdorf \cite{Tolksdorf} and Wang \cite{lihewang}.

The negative gradient flow of the energy functional \eqref{eq:p functional} takes the form of 
\begin{equation}\label{eq:divergence}
u_t= \Div (|\nabla u|^{p-2}\nabla u).
\end{equation}
H\"older estimates for the spatial gradient of weak solutions to \eqref{eq:divergence} were obtained by DiBenedetto and Friedman in \cite{DBF} (see also Wiegner \cite{Wiegner}), and we refer to the book  of DiBenedetto \cite{DiBenedetto} for an extensive overview on \eqref{eq:divergence} and more general cases. 

The equation above comes from a variational  interpretation of  the $p$-Laplacian operator. This is not the equation we study in this work. Our equation is motivated by the stochastic tug of war game interpretation of the $p$-Laplacian operator given by  Peres and Sheffield in \cite{PS}. Such time dependent stochastic games  will lead not to \eqref{eq:p div}, but rather the equation
\begin{equation}\label{eq:main}
u_t=|\nabla u|^{2-p} \Div (|\nabla u|^{p-2}\nabla u).
\end{equation}
This derivation is presented in Manfredi-Parviainen-Rossi  \cite{MPR}. The equation \eqref{eq:main} can also be written as
\begin{equation}\label{eq:main1}
u_t=(\delta_{ij} + (p-2)|\nabla u|^{-2}u_iu_j)u_{ij},
\end{equation}
where the summation convention is used. Through \eqref{eq:main1}, one can view the equation \eqref{eq:main} as the parabolic $p$-Laplacian equation in non divergence form. 

The majority of the previous work on elliptic and parabolic $p$-Laplace equation rely heavily on the variational structure of the equation. The equation \eqref{eq:main} does not have that structure. Therefore, we must take a completely different point of view using tools for equations in non-divergence form. To begin with, our notion of solution would be a viscosity solution instead of a solution in the sense of distributions. Thus, this work has hardly anything in common with the more classical results about regularity for $p$-Laplacian type equations. We use maximum principles and geometrical methods. 

Our work concerns the equation \eqref{eq:main} for the values of $p \in (1,+\infty)$. In this case, the existence and uniqueness of viscosity solutions to the initial-boundary value problems for \eqref{eq:main} have been established in Banerjee-Garofalo \cite{BG} and Does \cite{KD}, where they also proved the Lipschitz continuity in the spatial variables and studied the long time behavior of the viscosity solution. These properties were further studied in \cite{BG3,BG2, Juutinen2014} for \eqref{eq:main} or more general equations.  Manfredi, Parviainen and Rossi studied the equation \eqref{eq:main} as an asymptotic limit of certain mean value properties which are related to the tug-of-war game with noise originally described in \cite{PS}, when the number of rounds is bounded. One may find more results on the tug-of-war game with noise and the $p$-Laplacian operators in, e.g.,  \cite{KKK, LewMan, LS2015, MPR2012, Rossi2011,  RossiGAMES, Rudd2015}.

Even though all published work about the equation \eqref{eq:main} appeared only in recent years, we have seen an unpublished handwritten note by N. Garofalo from 1993 referring to this equation. In that note, there is a computation which leads to the result of Lemma \ref{lem:subsolution} in this paper. This is, up to our knowledge, the first time when it was recognized that the equation \eqref{eq:main} should have good regularization properties.

It is interesting to point out what our equation represents for $p=1$ and $p=+\infty$, even though these end-point cases are not included in our analysis. They appeared in the literature much earlier than \eqref{eq:main}. In these two cases, it is clear that the parabolic equation in non-divergence form \eqref{eq:main} is more important and better motivated than \eqref{eq:divergence}.

When $p=1$, the equation \eqref{eq:main} is the motion of the level sets of $u$ by its mean curvature, which has been studied by, e.g.,  Chen-Giga-Goto \cite{CGG}, Evans-Spruck \cite{ES1,ES2,ES3,ES4}, Evans-Soner-Souganidis \cite{ESS}, Ishii-Souganidis \cite{IS} and Colding-Minicozzi \cite{CM15}. A game of motion by mean curvature was introduced by Spencer \cite{Spencer} and studied by Kohn and Serfaty \cite{KSS}.

In another extremal case $p=+\infty$, it becomes the evolution governed by the infinity Laplacian operator. The infinity Laplacian operator  $\Delta_\infty$ defined by $\Delta_\infty u=\sum_{i,j}u_iu_iu_{ij}$ appears naturally when one considers absolutely minimizing Lipschitz extensions of a function defined on the boundary of a domain. Jensen \cite{Jensen}  proved that the absolute minimizer is the unique viscosity solution of the infinity Laplace equation $\Delta_\infty u=0$, of which the solutions are usually called infinity harmonic functions.  Savin in \cite{Savin} has shown that infinity harmonic functions are in fact continuously differentiable in the two dimensional case, and Evans-Savin \cite{ESavin} further proved the H\"older continuity of their gradient. Later, Evans-Smart \cite{ESmart} proved the everywhere differentiability of infinity harmonic functions in all dimensions. A game theoretical interpretation of this infinity Laplacian was given by Peres-Schramm-Sheffield-Wilson \cite{PSSW}. Finite difference methods for the infinity Laplace and $p$-Laplace equations were studied by Oberman \cite{Oberman2013}. The parabolic equation \eqref{eq:main} in this extremal case $p=\infty$ has been studied by, e.g., Juutinen-Kawohl \cite{JK} and Barron-Evans-Jensen \cite{BEJ}.

Our notion of solutions to \eqref{eq:main1} is the viscosity solution, which will be recalled in Definition \ref{den:viscosity} in Section \ref{sec: preliminaries}. For $1<p<\infty$, one observes that
\[
 \min(p-1,1)I\le \delta_{ij}+(p-2)q_iq_i |q|^{-2}\le  \max(p-1,1)I\quad\mbox{for all }q\in\R^n\setminus\{0\},
\]
where $I$ is the $n\times n$ identity matrix. Therefore, the equation \eqref{eq:main1} is uniformly parabolic. It follows from the regularity theory of Krylov-Safonov \cite{KS} that the viscosity solution $u$ of \eqref{eq:main1} is H\"older continuous in the space-time variables. As mentioned earlier,  Banerjee-Garofalo \cite{BG} and Does \cite{KD} proved that the solution $u$ is Lipschitz continuous in the spatial variables. Whether or not the spatial gradient $\nabla u$ is H\"older continuous was left as an interesting open question.

In this paper, we answer this question and prove the following interior H\"older estimates for the spatial gradient of viscosity solutions to \eqref{eq:main1}.

\begin{thm}\label{thm:mainholdergradient}
 Let $u$ be a viscosity solution of \eqref{eq:main1} in $Q_1$, where $1<p<\infty$. Then there exist two constants $\alpha\in (0,1)$ and $C>0$, both of which depends only on $n$ and $p$, such that 
\[
 \|\nabla u\|_{C^\alpha(Q_{1/2})}\le C\|u\|_{L^\infty(Q_1)}.
\]
Also, there holds
\[
\sup_{(x,t), (x,s)\in Q_{1/2}}\frac{|u(x,t)-u(x,s)|}{|t-s|^{\frac{1+\alpha}{2}}}\le C\|u\|_{L^\infty(Q_1)}.
\]
\end{thm}
Here, $Q_r=B_r\times(-r^2,0]$ is denoted as the standard parabolic cylinder, where $r>0$ and $B_r\subset\R^n$ is the ball of radius $r$ centered at the origin. Combining those two estimates in Theorem \ref{thm:mainholdergradient}, we have that
\[
|u(y,s)-u(x,t)-\nabla u(x,t)\cdot (y-x)|\le C\|u\|_{L^\infty(Q_1)} (|y-x|+\sqrt{|t-s|})^{1+\alpha}
\]
for all $(y,s), (x,t)\in Q_{1/2}$.

The equation \eqref{eq:main1} is quasi-linear, and $(\delta_{ij} + (p-2)|\nabla u|^{-2}u_iu_j)$ can be viewed as the coefficients of the equation. Note that these coefficients  have a singularity when $\nabla u=0$. This %singularity plays a role around the points where $\nabla u=0$ and 
is what causes the main difficulties in the proof of our main result. The only thing in common with previous proofs of $C^{1,\alpha}$ regularity with equations of $p$-Laplacian type is perhaps the general outline of steps necessary for the proof. The oscillation of the gradient is reduced in a shrinking sequence of parabolic cylinders. The iterative step is reduced to a dichotomy between two cases: either the value of the gradient $\nabla u$ stays close to a fixed vector $e$ for most points $(x,t)$ (in measure), or it does not. The way each of these two cases is resolved (which is the key of the proof), follows a new idea. Traditionally, the variational structure of the equation played a crucial role in the resolution of each of these two cases. The key ideas in this paper are contained in Section 4, especially in Lemmas \ref{lem:osc2} and \ref{l:gradientarounde2smalloscillation}. Lemma \ref{l:gradientarounde2smalloscillation} allows us to apply a recent result by Yu Wang \cite{wang} (which is the parabolic version of a result by Savin \cite{Savin07}) to resolve one of the two cases in the dichotomy.

In the process of proving Lemma \ref{l:gradientarounde2smalloscillation}, we obtain Lemma  \ref{l:timeslices2allcylinder} which is a general property of solutions to uniformly parabolic equations and may be interesting by itself. It states that an upper bound on the oscillation $\osc_{x \in B_1} u(x,t)$ for every fixed $t \in [a,b]$ implies an upper bound in space-time for $\osc_{(x,t) \in B_1 \times [a,b]} u(x,t)$.

In future work \cite{IJS}, we plan to adapt the method presented in this paper to obtain the H\"older continuity of $\nabla u$ for the following generalization of \eqref{eq:main}:
\[ u_t = |\nabla u|^\kappa \Delta_p u.\]
Here $\kappa$ is an arbitrary power in the range $\kappa \in (1-p,+\infty)$. The equation generalizes both the classical (scalar) parabolic $p$-Laplacian equation in divergence form \eqref{eq:divergence} and in non divergence form \eqref{eq:main}.

This paper is organized as follows. In Section \ref{sec: preliminaries}, we start by recalling some well-known regularity results for solutions of uniformly parabolic equations which will be used in our proof of Theorem \ref{thm:mainholdergradient}, as well as the definition and two properties of the viscosity solutions of \eqref{eq:main1}.  We then introduce a regularization procedure for \eqref{eq:main1}. In Section \ref{sec:lip}, we will establish Lipschitz estimates for the solutions $u$ of its regularized problem. The result of Section 3 is not new, but we present a new proof within our context. In Section \ref{sec:holder}, we obtain the H\"older estimates for $\nabla u$, which is the most technically challenging part and the core of this paper. Finally, Theorem \ref{thm:mainholdergradient} will follow from approximation arguments, whose details will be presented in Section \ref{sec:appr}. 

\section{Preliminaries}\label{sec: preliminaries}
In this section, we recall some known regularity results for solutions of linear uniformly parabolic equations with measurable coefficients: 
\begin{equation}\label{eq:linear parabolic}
 u_t - a_{ij}(x,t) \partial_{ij} u = 0 \quad\text{ in } Q_1,
\end{equation}
where $a_{ij}(x,t)$ is uniformly parabolic, i.e. there are constants $0 < \lambda \leq \Lambda<\infty$ such that
\begin{equation}\label{eq:ellipticity} 
 \lambda I \leq a_{ij}(x,t) \leq \Lambda I \quad\text{ for all } (x,t)\in Q_1.
\end{equation}
The first two in the below are the weak Harnack inequality and local maximum principle due to Krylov and Safonov. For their proofs, we refer to the lecture notes by Imbert and Silvestre \cite{ImbertS}. 

\begin{thm}[Weak Harnack inequality]\label{thm:weak harnack}
 Let $u\in C(Q_1)$ be a non negative supersolution of \eqref{eq:linear parabolic} satisfying \eqref{eq:ellipticity}. Then there exist two positive constants $\theta$ (small) and $C$ (large), both of which depend only on $n,\lambda$ and $\Lambda$, such that
\[
 \|u\|_{L^\theta(Q^*_{1/2})}\le C \inf_{Q_{1/2}}u,
\]
where $Q_{1/2}^*=B_{1/2}\times (-1,-3/4)$.
\end{thm}

\begin{thm}[Local maximum principle]\label{thm:local max principle}
 Let $u\in C(Q_1)$ be a subsolution of \eqref{eq:linear parabolic} satisfying \eqref{eq:ellipticity}. For every $\gamma>0$, there  exists a positive constant  $C$  depending only on $\gamma, n,\lambda$ and $\Lambda$, such that
\[
 \sup_{Q_{1/2}}u\le C \|u^+\|_{L^\gamma(Q_1)},
\]
where $u^+=\max(u,0)$.
\end{thm}

The exact statement which we will use regarding to improvement of oscillation for supersolutions of \eqref{eq:linear parabolic}  is of the following form.

\begin{prop}[Improvement of oscillation]\label{prop:io}
 Let $u\in C(Q_1)$ be a non negative supersolution of \eqref{eq:linear parabolic} satisfying \eqref{eq:ellipticity}. For every $\mu\in (0,1)$, there exist two positive constants $\tau$ and $\gamma$, where $\tau$ depends only on $\mu$ and $n$, and $\gamma$ depends only on $\mu, n,\lambda$ and $\Lambda$, such that if
\[
|\{(x,t)\in Q_1: u\ge 1\}|>\mu |Q_1|,
\]
then
\[
u\ge \gamma\quad\mbox{in }Q_\tau.
\]
\end{prop}
\begin{proof}
First of all, we can choose $\tau>0$ small such that $1/\tau$ is an integer, and for $\Omega:= B_{1-6\tau}\times (-1,-9\tau^2]$, there holds 
\begin{align*}
|\{(x,t)\in \Omega: u\ge 1\}|&\ge |\{(x,t)\in Q_1: u\ge 1\}|-|Q_1\setminus\Omega|\\
&>\mu|Q_1|-C(n)\tau\\
&>\frac{\mu}{2} |Q_1|,
\end{align*}
where $C(n)$ is some positive constant depending on $n$. Note that this choice of $\tau$ depends on $\mu$ and $n$ only. Then, we use $N$ cylinders $Q^{(j)}\subset Q_1$, $Q^{(j)}\cap\Omega\neq\emptyset$, $j=1,2,\cdots, N$, all of which are of the same size as $Q_\tau$, to cover $\Omega$ in the way of covering the slices $B_{1-6\tau}\times(-1+(k-1)\tau^2, -1+k\tau^2]$ one by one for $k=1,2,\cdots,1/\tau^2-9$. This integer $N$ depends only on $\tau$ and $n$. Then there exists at least one  cylinder, which is denoted as $Q_\tau(x_0,t_0)=Q_\tau+(x_0,t_0)$ for some $(x_0,t_0)\in B_{1-5\tau}\times (-1+\tau^2,-8\tau^2]$, such that
\[
 |\{(x,t)\in Q_\tau(x_0,t_0): u\ge 1\}|\ge \frac{\mu}{2N} |Q_1|,
\]
since otherwise,
\begin{align*}
 |\{(x,t)\in \Omega: u\ge 1\}|\le |\bigcup_{j=1}^N \{(x,t)\in Q^{(j)}: u\ge 1\}| \le \sum_{j=1}^N| \{(x,t)\in Q^{(j)}: u\ge 1\}|<\frac{\mu}{2} |Q_1|,
\end{align*}
which is a contradiction. By Theorem \ref{thm:weak harnack}, there exists $m>0$ depending only on $\mu, n,\lambda$ and $\Lambda$ such that
\[
u\ge m\quad\mbox{in }Q_\tau(x_0,t_0+2\tau^2).
\]
Then by applying Lemma 4.1 in \cite{FS} to $m-u$, we obtain that 
\[
u\ge\gamma\quad\mbox{in }Q_\tau
\]
for some positive $\gamma$ depending only on $\mu, n,\lambda$ and $\Lambda$.
\end{proof}

A consequence of Theorem \ref{thm:weak harnack} and Theorem \ref{thm:local max principle} is the following interior H\"older estimate by Krylov and Safonov \cite{KS}. 

\begin{thm}[Interior H\"older estimates]\label{thm:interior holder}
 Let $u\in C(Q_1)$ be a solution of \eqref{eq:linear parabolic} satisfying \eqref{eq:ellipticity}. Then there exist two positive constants $\alpha$ (small) and $C$ (large), both of which depend only on $n,\lambda$ and $\Lambda$, such that
\[
 \|u\|_{C^\alpha(Q_{1/2})}\le C \osc_{Q_1} u
\]
\end{thm}

Here, we write $\osc_Q\, u:= \sup_Q u - \inf_Q u$. Note that by adding or subtracting an appropriate constant, the estimate in the previous theorem is equivalent to
\[ \|u\|_{C^\alpha(Q_{1/2})}\le C \|u\|_{L^\infty(Q_1)}.\]

Meanwhile, we shall also use a boundary regularity property.  For two real numbers $a$ and $b$, we denote $$a\vee b=\max(a,b),\  a\wedge b=\min(a,b).$$ We also denote $$\partial_p Q_r=(\partial B_r\times(-r^2,0))\cup (B_r\times \{(x,t): t=-r^2\})$$ as the so-called parabolic boundary of $Q_r$.

\begin{prop}[Boundary estimates]\label{prop:boundary regularity}
 Let $u\in C(\overline Q_1)$ be a solution of \eqref{eq:linear parabolic} satisfying \eqref{eq:ellipticity}. Let $\varphi:=u|_{\partial_p Q_1}$ and let $\rho$ be a modulus of continuity of $\varphi$. Then there exists another modulus of continuity $\rho^*$ depending only on $n,\lambda,\Lambda, \rho, \|\varphi\|_{L^\infty(\partial_p Q_1)}$ such that
\[
 |u(x,t)-u(y,s)|\le \rho^*(|x-y|\vee\sqrt{|t-s|})
\]
for all $(x,t), (y,s)\in \overline Q_1$.
\end{prop}
The above proposition is an adaptation of Proposition 4.14 in \cite{CC} for parabolic equations, whose proof will be given in Appendix \ref{sec:appB}. 

Another useful result is the $W^{2,\delta}$ estimate for parabolic equations, which can be found in Theorem 1.9 and Theorem 2.3 of Krylov \cite{Krylov}. Such estimates were first discovered by F.-H. Lin \cite{Lin86} for elliptic equations. 
\begin{thm}[$W^{2,\delta}$ estimates]\label{thm:W 2 delta}
 Let $u\in C(\overline Q_1)\cap C^2(Q_1)$ be a solution of \eqref{eq:linear parabolic} satisfying \eqref{eq:ellipticity}. Then there exist two positive constants $\delta$ (small) and $C$ (large), both of which depend only on $n,\lambda$ and $\Lambda$, such that
\[
 \|\nabla u\|_{L^\delta(Q_{1})}+ \|\nabla^2 u\|_{L^\delta(Q_{1})}\le C \|u\|_{L^\infty(\partial_p Q_1)}.
\]
\end{thm}

The last one we will use in this paper is a regularity estimate for small perturbation solutions of  fully nonlinear parabolic equations, which was proved by Wang \cite{wang}. Such estimates were first proved by Savin \cite{Savin07} for fully nonlinear elliptic equations.

\begin{thm}[Regularity of small perturbation solutions] \label{l:smallness to regularity}
Let $u$ be a smooth solution of \eqref{eq:main va} in $Q_1$. For each $\gamma\in (0,1)$, there exist two positive constants $\eta$ (small) and $C$ (large), both of which depends only on $\gamma, n$ and $p$, such that if $|u(x,t)-L(x)|\le\eta$ in $Q_1$ for some linear function $L$ of $x$ satisfying $1/2\le|\nabla L|\le 2$, then 
\[
\|u-L\|_{C^{2,\gamma}(Q_{1/2})}\le C.
\]
\end{thm}
\begin{proof}
Since $L$ is a solution of \eqref{eq:main va}, the conclusion follows from Corollary 1.2 in \cite{wang}.
\end{proof}

Now let us recall the definition of viscosity solutions to \eqref{eq:main1} (see Definition 2.3 in \cite{BG}). 
\begin{defn}\label{den:viscosity}
An upper (lower, resp.) semi-continuous function $u$ in $Q_1$ is called a viscosity subsolution (supersolution, resp.) of \eqref{eq:main1} in $Q_1$ if for every $\varphi\in C^2(Q_1)$, $u-\varphi$ has a local maximum (minimum, resp.) at $(x_0,t_0)\in Q_1$, then
\[
\varphi_t\le (\ge, resp.) \Delta \varphi + (p-2)|\nabla \varphi|^{-2}\varphi_i\varphi_j\varphi_{ij}
\]
at $(x_0,t_0)$ when $\nabla \varphi(x_0,t_0)\neq 0$, and
\[
\varphi_t\le (\ge, resp.) \Delta \varphi + (p-2)q_iq_j\varphi_{ij}
\]
for some  $q\in\overline B_1\subset\R^n$  at $(x_0,t_0)$  when $\nabla \varphi(x_0,t_0)= 0$.

A function $u\in C(Q_1)$ is called a viscosity solution of \eqref{eq:main1}, if it is both a viscosity subsolution and a viscosity supersolution.
\end{defn}

In order to circumvent the inconveniences of the lack of smoothness of viscosity solutions, we choose to approximate the equation \eqref{eq:main1} with a regularized problem. For $\va>0$, let $u$ be smooth and  satisfy that
\be\label{eq:main va}
u_t=a_{ij}(\nabla u) u_{ij}\quad\mbox{in }Q_1,
\ee
where 
\begin{equation}\label{eq:aij}
 a_{ij} (q)=\delta_{ij}+(p-2)\frac{q_iq_j}{|q|^2+\va^2}\quad\mbox{for }q\in\R^n.
\end{equation}
This equation \eqref{eq:main va} is uniformly parabolic and has smooth solutions for all $\eps>0$. Such regularization techniques have been used before for the $p$-Laplace equation in several contexts. For example, see \cite{BG, ES1, Lewis}. We will obtain a priori estimates that are independent of $\eps$ and finally show that they apply to the original equation \eqref{eq:main1} through approximations. 

In the step of approximation, we will use the following two properties on the viscosity solutions of \eqref{eq:main1}. The first one is the comparison principle, which can be found in Theorem 3.2 in \cite{BG}.

\begin{thm}[Comparison principle]\label{thm:comparison principle}
Let $u$ and $v$ be a viscosity subsolution and a viscosity supersolution of \eqref{eq:main1} in $Q_1$, respectively. If $u\le v$ on $\partial_pQ_1$, then $u\le v$ in $\overline Q_1$.
\end{thm}
The second one is the stability of viscosity solutions of \eqref{eq:main1}. 
\begin{thm}[Stability]\label{thm:stability}
Let $\{u_k\}$ be a sequence of viscosity subsolutions of \eqref{eq:main va} in $Q_1$ with $\va_k\ge 0$ that $\va_k\to 0$, and $u_k$ converge locally uniformly to $u$ in $Q_1$. Then $u$ is a viscosity subsolution of \eqref{eq:main1} in $Q_1$.
\end{thm}
\begin{proof}
We refer to the proof of Theorem 2.7 in \cite{ES1} or the second paragraph of the proof of Theorem 4.2 in \cite{ES1}.
\end{proof}

To summarize, we would like to mention what each of these results in this section will be used for in our proof of Theorem \ref{thm:mainholdergradient}. The local maximum principle in Theorem \ref{thm:local max principle} and the $W^{2,\delta}$ estimates in Theorem \ref{thm:W 2 delta} will be used to prove Lipschitz estimates. The form of improvement of oscillation in Proposition \ref{prop:io}, the interior H\"older estimates in Theorem \ref{thm:interior holder} and the regularity of small perturbation solutions in Theorem \ref{l:smallness to regularity} are the key ingredients in our proof of the H\"older gradient estimates. The boundary estimates in Proposition \ref{prop:boundary regularity}, as well as the comparison principle in Theorem \ref{thm:comparison principle} and the stability property in Theorem \ref{thm:stability} will only be used in the technical approximation step, which do not affect the proof of the a priori estimates.

\section{Lipschitz estimates in spatial variables}\label{sec:lip}

 The interior Lipschitz estimate for solutions of \eqref{eq:main va} in spatial variables was essentially obtained before by Does \cite{KD}. Here, we will  provide an alternative proof. Our proof appears much shorter since it uses Theorem \ref{thm:local max principle} and Theorem \ref{thm:W 2 delta}, whereas, the proof given by Does \cite{KD} is based on the Bernstein technique and uses only elementary
tools.

The following auxiliary lemma follows from a direct calculation. We postpone its proof to Appendix \ref{A}.
\begin{lem}\label{lem:subsolution}
For a smooth solution $u$ of \eqref{eq:main va} and $\varphi:=(|\nabla u |^2+\va^2)^\frac{p}{2}$ we have
\begin{align*}
\big(\partial_t -a_{ij}(\nabla u) \partial_{ij}\big)\varphi\le 0,
\end{align*}
where $a_{ij}(\nabla u)$ is given in \eqref{eq:aij}.
\end{lem}

We now present the interior Lipschitz estimate.

\begin{thm}\label{thm:lip}
Let $u$ be a smooth solution of \eqref{eq:main va} in $Q_1$. Then there exists a positive constant $C$ depending only on $n$ and $p$ such that 
\[
 \|\nabla u\|_{L^\infty(Q_{1/2})}\le C(\|u\|_{L^\infty(Q_1)}+\va).
\]
\end{thm}

\begin{proof}
 Since $u$ satisfies \eqref{eq:main va}, it follows from Theorem \ref{thm:W 2 delta} that there exist two positive constants $\delta$ (small) and $C$ (large) both of which depend only on $n$ and $p$ such that 
\[
 \|\nabla u\|_{L^\delta(Q_{3/4})}\le  C\|u\|_{L^\infty(Q_1)}.
\]
Let $\varphi:=(|\nabla u |^2+\va^2)^\frac{p}{2}$. By Lemma \ref{lem:subsolution} and Theorem \ref{thm:local max principle}, we have
\[
 \begin{split}
  \|\varphi\|_{L^\infty(Q_{1/2})}&\le C \|\varphi\|_{L^{\delta/p} (Q_{3/4})} \le C (\|\nabla u\|^p_{L^\delta(Q_{3/4})}+\va^p)\le  C(\|u\|^p_{L^\infty(Q_1)}+\va^p).
 \end{split}
\]
It follows that
\[
 \|\nabla u\|_{L^\infty(Q_{1/2})}\le C(\|u\|_{L^\infty(Q_1)}+\va).
\]
\end{proof}

\section{H\"older estimates for the spatial gradients}\label{sec:holder}

In this section, we shall prove the H\"older estimate of $\nabla u$ at $(0,0)$. By Theorem \ref{thm:lip} and normalization, we assume that $|\nabla u|\le 1$. The idea is the following. First, we show that if the projection of $\nabla u$ onto the direction $e\in\mathbb S^{n-1}$ is away from $1$ in a positive portion of $Q_1$, then $\nabla u\cdot e$ has improved oscillation in $Q_{\tau}$ for some $\tau>0$. 

Then we analyze according to the following dichotomy:

\begin{itemize}
\item If we can keep scaling around $(0,0)$ and iterate infinitely many times in all directions $e\in\mathbb S^{n-1}$, then it leads to the H\"older continuity of $\nabla u$  at $(0,0)$. 

\item If the iteration stops at, say, the $k$-th step  in some direction $e\in\mathbb S^{n-1}$. This means that $\nabla u$ is close to some fixed vector in a large portion of $Q_{\tau^k}$. We then prove that $u$ is close to some linear function, and the H\"older continuity of $\nabla u$ will follow from Theorem \ref{l:smallness to regularity}.
\end{itemize}

\subsection{Improvement of oscillation}
Since $\nabla u$ is a vector, we shall first obtain an improvement of oscillation for $\nabla u$ projected to an arbitrary direction $e\in\mathbb S^{n-1}$.

\begin{lem}\label{lem:osc2}
Let $u$ be a smooth solution of \eqref{eq:main va} such that $|\nabla u|\le 1$ in $Q_1$. 
For every $0<\ell<1$, $\mu>0$, there exists $\tau>0$ depending only on $\mu$ and $n$, and there exists $\delta>0$ depending only on $n,p,\mu$ and $\ell$ such that for arbitrary $e\in\mathbb{S}^{n-1}$, if
\begin{equation}\label{eq:proj on e positive}
|\{(x,t)\in Q_1: \nabla u\cdot e\le \ell\}|> \mu |Q_1|,
\end{equation}
then
\[
\nabla u\cdot e< 1-\delta\quad\mbox{in }Q_{\tau}.
\]
\end{lem}
\begin{proof}
Let $a_{ij}$ be as in \eqref{eq:aij}, and denote
\[
a_{ij,m}=\frac{\partial a_{ij}}{\partial q_m}.
\]
Differentiating \eqref{eq:main va} in $x_k$, we have
\[
(u_k)_t=a_{ij}\big(u_{k})_{ij}+a_{ij,m}u_{ij} (u_{k})_m.
\]
Then 
\[
 (\nabla u\cdot e-\ell)_t=a_{ij}\big(\nabla u\cdot e-\ell)_{ij}+a_{ij,m}u_{ij} (\nabla u\cdot e-\ell)_m.
\]
Therefore,  for
\[
v=|\nabla u|^2,
\]
we have
\[
v_t=a_{ij}v_{ij}+a_{ij,m}u_{ij}v_m-2a_{ij}u_{ki}u_{kj}.
\] 
For $\rho=\ell/4$, let
\[
 w=(\nabla u\cdot e-\ell+\rho |\nabla u|^2)^+.
\]
Then in the region $\Omega_+=\{(x,t)\in Q_1: w>0\}$, we have
\[
 w_t=a_{ij} w_{ij} + a_{ij,m}u_{ij}w_m-2\rho a_{ij}u_{ki}u_{kj}.
\]
Since $|\nabla u|>\ell/2$ in $\Omega_+$, we have
\[
 |a_{ij,m}|\le \frac{4|p-2|}{\ell}\quad\mbox{in }\Omega_+.
\]
By Cauchy-Schwarz inequality, it follows that
\[
  w_t\le a_{ij} w_{ij} + \frac{c_0}{\rho\ell^2}|\nabla w|^2\quad\mbox{in }\Omega_+,
\]
for some constant $c_0>0$ depending only on $p$. Therefore, it satisfies in the viscosity sense that
\[
  w_t\le a_{ij} w_{ij} + \frac{c_0}{\rho\ell^2}|\nabla w|^2\quad\mbox{in }Q_1.
\]
We can choose $c_1$ such that if we let
\[
 W=1-\ell+\rho,\quad\nu=\frac{c_1}{\rho \ell^2},
\]
and
\[
\overline w =\frac{1}{\nu}(1-e^{\nu(w -W )}),
\]
then we have
\[
   \overline w_t\ge a_{ij} \overline w_{ij} \quad\mbox{in }Q_1
\]
in the viscosity sense. Since $W\ge \sup_{Q_1}w$, then $\overline w\ge 0$ in $Q_1$. 

If $\nabla u\cdot e\le \ell$, then $\overline w\ge (1-e^{\nu (\ell-1)})/\nu$. Therefore, it follows from the assumption that
\[
|\{(x,t)\in Q_1: \overline w\ge (1-e^{\nu (\ell-1)})/\nu\}|> \mu |Q_1|.
\]
By Proposition \ref{prop:io}, there exist $\tau>0$ depending only $\mu$ and $n$, and $\gamma>0$ depending only on $\mu, \ell, n$ and $p$ such that
\[
\overline w\ge \gamma\quad\mbox{in }Q_\tau.
\]
Meanwhile, since $w\le W$, we have
\[
\overline w\le W-w.
\]
This implies that
\[
W-w\ge\gamma\quad\mbox{in }Q_{\tau}.
\]
Therefore, we have
\[
 \nabla u\cdot e+\rho|\nabla u|^2 \le 1+\rho -\gamma\quad\mbox{in }Q_{\tau}.
\]
Since $| \nabla u\cdot e|\le |\nabla u|$, we have
\[
 \nabla u\cdot e+\rho ( \nabla u\cdot e)^2\le 1+\rho - \gamma\quad\mbox{in }Q_{\tau}.
 \]
Therefore,
 \[
  \nabla u\cdot e \le \frac{-1+\sqrt{1+4\rho (1+\rho - \gamma)}}{2\rho}\le 1-\delta\quad\mbox{in }Q_{\tau}
 \]
 for some $\delta>0$ depending only on $p, \mu, \ell, n$. 
 \end{proof}

The statement of Lemma \ref{lem:osc2} can be illustrated in Figure \ref{fig:1}.
\begin{figure}[h!]
  \caption{Improvement of oscillation for $\nabla u\cdot e$.}\label{fig:1}
  \centering
   \includegraphics[width=11cm]{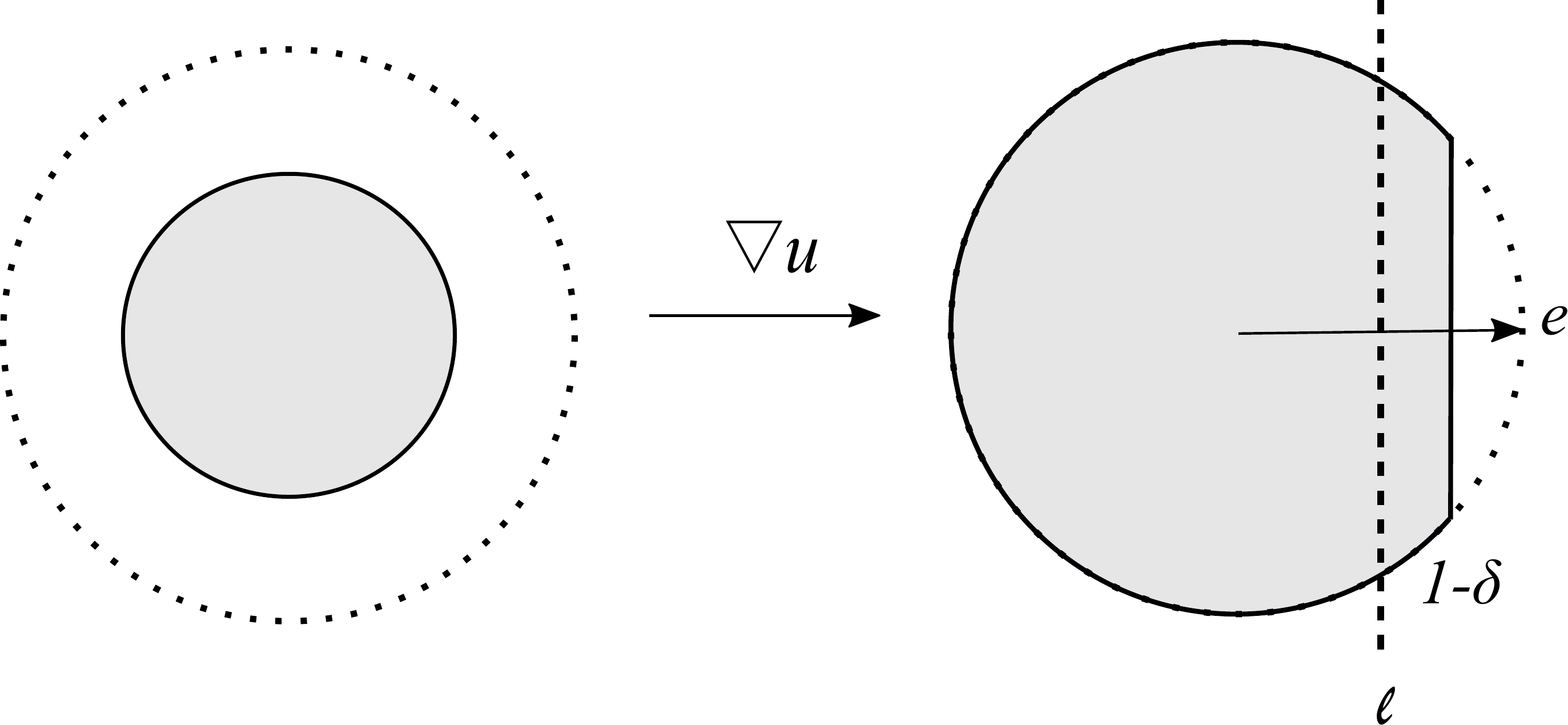}
\end{figure}

If the condition \eqref{eq:proj on e positive} is satisfied in all the directions $e\in\mathbb S^{n-1}$, then we obtain the improvement of oscillation for all $\nabla u\cdot e$, which lead to the improvement of oscillation for $|\nabla u|$. See Figure \ref{fig:2} and Corollary \ref{cor:osc}.
\begin{figure}[h!]
  \caption{Improvement of oscillation for $|\nabla u|$.}\label{fig:2}
  \centering
   \includegraphics[width=11cm]{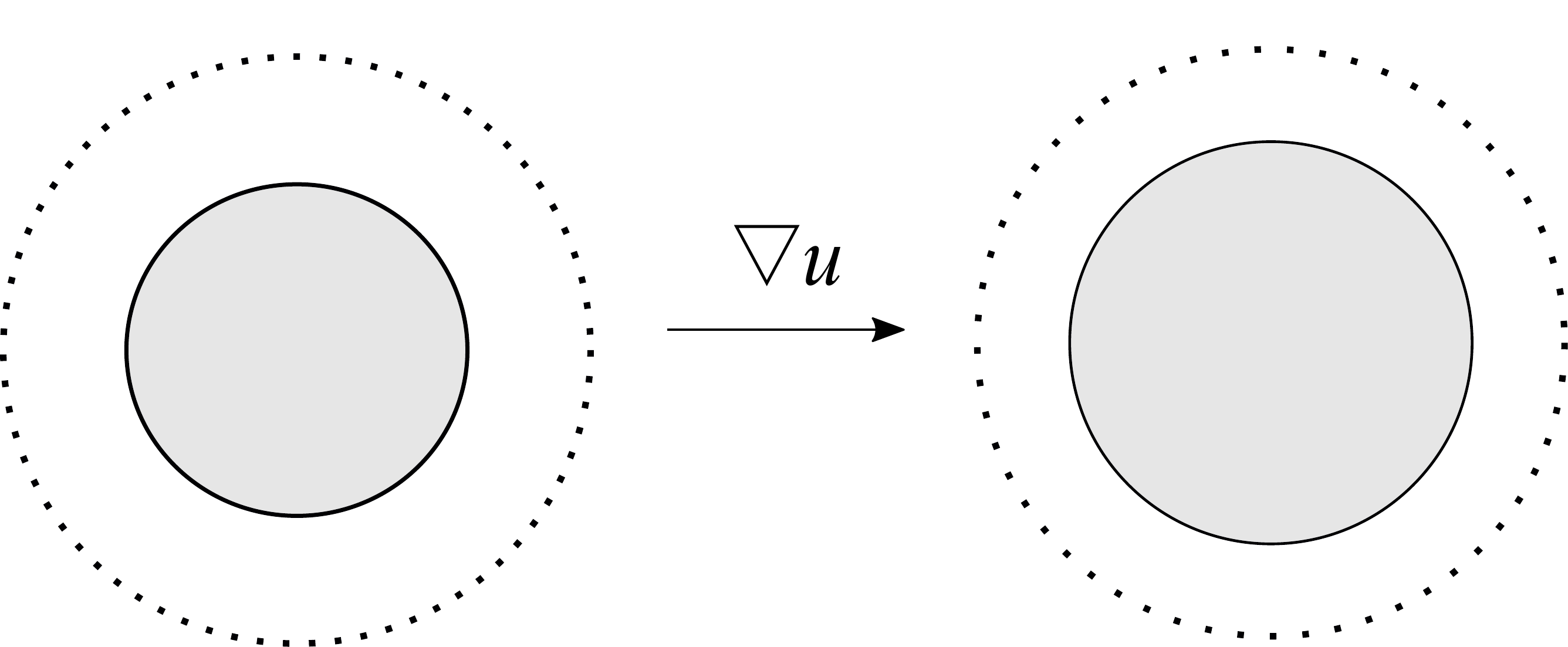}
\end{figure}

\begin{cor}\label{cor:osc}
Let $u$ be a smooth solution of \eqref{eq:main va} such that $|\nabla u|\le 1$ in $Q_1$. For every  $0<\ell<1$, $\mu>0$, there exist $\tau\in (0,1/4)$ depending only on $\mu$ and $n$, and  $\delta>0$ depending only on $n,p,\mu, \ell$, such that for every nonnegative integer $k$, if
\begin{equation}\label{eq:stopping condition}
|\{(x,t)\in Q_{\tau^i}: \nabla u\cdot e\le \ell (1-\delta)^i\}|> \mu |Q_{\tau^i}|\quad\mbox{for all }e\in\mathbb S^{n-1}\mbox{ and }i=0,\cdots, k,
\end{equation}
then
\[
|\nabla u|< (1-\delta)^{i+1}\quad\mbox{in }Q_{\tau^{i+1}} \ \mbox{  for all }i=0,\cdots, k.
\]
\end{cor}
\begin{proof}
When $i=0$, it follows from Lemma \ref{lem:osc2} that $\nabla u\cdot e< 1-\delta$ in $Q_{\tau}$ for all $e\in\mathbb S^{n-1}$. This implies that $|\nabla u|<1-\delta$ in $Q_\tau$.

Suppose this corollary holds for $i=0,\cdots, k-1$. We are going prove it for $i=k$. Let \[v(x,t)=\frac{1}{\tau^k (1-\delta)^k}u(\tau^k x, \tau^{2k} t).\] Then $v$ satisfies
\[
v_t=\Delta v+(p-2)\frac{v_iv_j}{|\nabla v|^2+\eps (1-\delta)^{-2k}} v_{ij}\quad\mbox{in } Q_1.
\]
By the induction hypothesis, we also know that $|\nabla v|\le 1$ in $Q_1$, and
\[
|\{(x,t)\in Q_{1}: \nabla v\cdot e\le \ell \}|> \mu |Q_{1}| \quad\mbox{for all }e\in\mathbb S^{n-1}.
\]
Therefore, by Lemma \ref{lem:osc2} we have
\[
\nabla v\cdot e \le 1-\delta\quad\mbox{in }Q_{\tau} \quad\mbox{for all }e\in\mathbb S^{n-1}.
\]
Hence, $|\nabla v|\le 1-\delta$ in $Q_\tau$. Consequently,
\[
|\nabla u|< (1-\delta)^{k+1}\quad\mbox{in }Q_{\tau^{k+1}}. 
\]
\end{proof}

\subsection{Using the small oscillation}

Unless $|\nabla u(0,0)|=0$, the condition in \eqref{eq:stopping condition} will fail to be satisfied after finitely many steps of scaling in some direction $e\in\mathbb S^{n-1}$, in which we will then show that $u$ is close to some linear function so that Theorem \ref{l:smallness to regularity} can be applied. See Lemma \ref{l:gradientarounde2smalloscillation} and Figure \ref{fig:3}.

Before that, we need a lemma which states that for a solution of a uniformly parabolic linear equation, if its oscillation in space is uniformly small in every time slice, then its oscillation in the space-time is also small.

\begin{lem} \label{l:timeslices2allcylinder}
 Let $u\in C(\overline Q_1)$ be a solution of \eqref{eq:linear parabolic} satisfying \eqref{eq:ellipticity} and $A$ is a positive constant. 
Assume that for all $t \in [-1,0]$, we have
\[ \osc_{B_1} u(\cdot, t) \leq A,\]
then
\[ \osc_{Q_1} u \leq C A,\]
where $C$ is a positive constant  depending only on $\Lambda$ and the dimension $n$.
\end{lem}

\begin{proof}
Let $\overline w(x,t) = \overline a+ 5n\Lambda A t + 2 A |x|^2$, where $a$ is chosen so that $\overline w(\cdot,-1) \geq u(\cdot,-1)$ and $\overline w(\bar x, -1) = u(\bar x,-1)$ for some $\bar x\in\overline B_1$. If $\bar x\in\partial B_1$, then
\[
2A= \overline w(\bar x,-1)-\overline w(0,-1)\le u(\bar x,-1)-u(0,-1)\le \osc_{B_1} u(\cdot-1) \leq A,
\]
which is impossible. Therefore, $\bar x\in B_1$.

We claim that \[\overline w \geq u \quad\text{in } Q_1.\] 
If not, let $m=-\inf_{Q_1} (\overline w-u)>0$ and $(x_0,t_0)\in\overline Q_1$ be such that $m=u(x_0,t_0)-\overline w(x_0,t_0)$. By the choice of $\bar a$, we know that $t_0>-1$. Since $\overline w+m\ge u$ in $Q_1$, $\overline w(x_0,t_0)+m=u(x_0,t_0)$ and $\osc_{B_1} u(\cdot, t_0)\le A$, by the same reason in the above, we have $x_0\in B_1$. Therefore, we have that
\[
(\overline w+m)_t - a_{ij}(x,t) \partial_{ij} (\overline w+m) \le 0\quad\text{at }(x_0,t_0).
\]
This leads to 
\[
5n\Lambda A\le 4A\cdot  Tr(a_{ij})\le 4n\Lambda A,
\]
which is impossible. This proves the claim.

Similarly, one can show that for $\underline w(x,t) = \underline a - 5n\Lambda A t - 2 A |x|^2$, we have
\[
\underline w \leq u \quad\text{in } Q_1,
\]
where $\underline a$ is chosen so that $\underline w(\cdot,-1) \leq u(\cdot,-1)$ and $\underline w(\underline x,-1) = u(\underline x,-1)$ for some $\underline x\in B_1$. 

Meanwhile, since
\[
\overline w(\bar x,-1)-\underline w(\underline x,-1)=u(\bar x,-1)-u (\underline x,-1)\le osc_{B_1}u(\cdot,-1)\le A, 
\]
we have
\[
\bar a-\underline a\le (10n\Lambda+1)A.
\]
Therefore, we have
\[
\osc_{Q_1} u\le \sup_{Q_1}\overline w-\inf_{Q_1}\underline w\le \bar a-\underline a+4A=(10 n\Lambda +5)A.
\]
\end{proof}

\begin{figure}[h!]
  \caption{When $|\nabla u(0,0)|\neq 0$.}\label{fig:3}
  \centering
   \includegraphics[width=11cm]{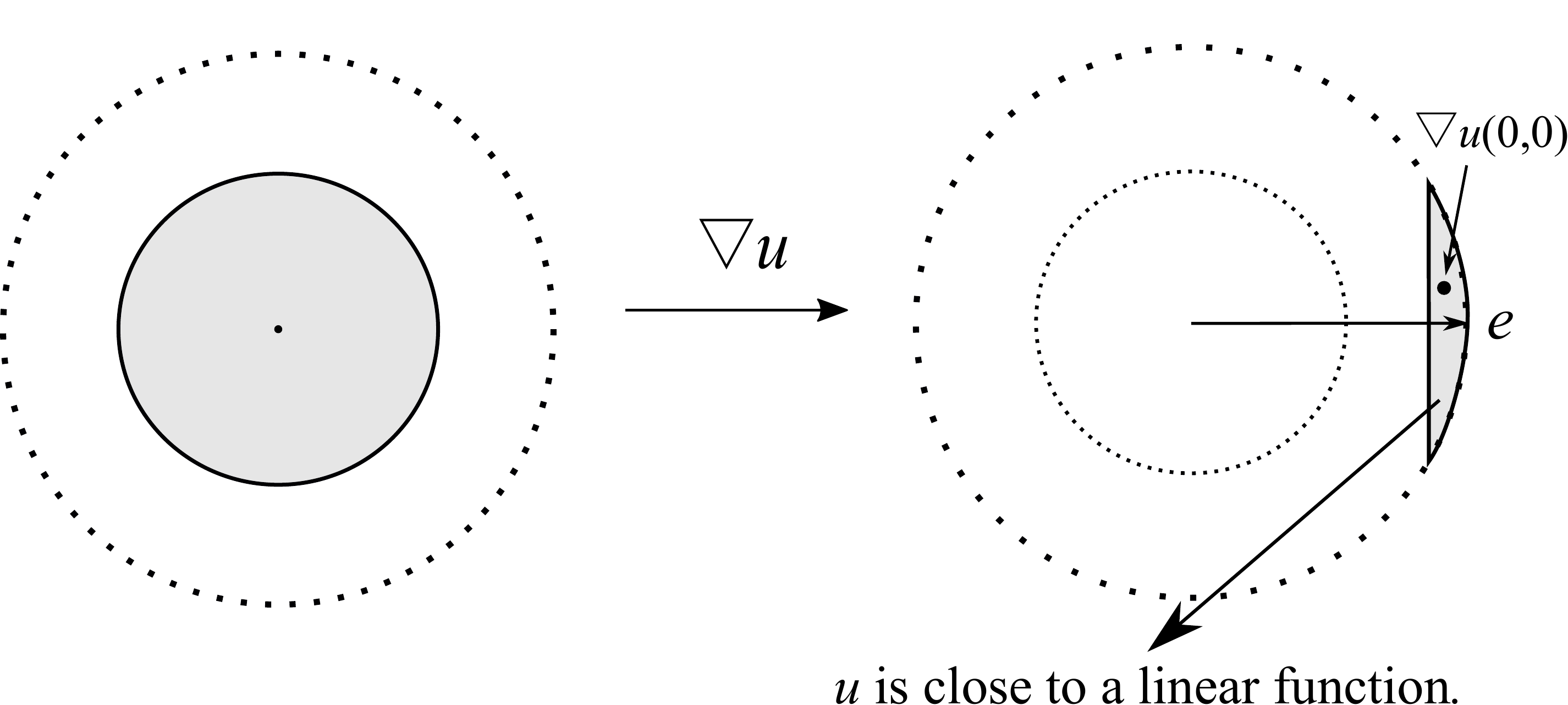}
\end{figure}

\begin{lem} \label{l:gradientarounde2smalloscillation}
Let $\eta$ be a positive constant and $u$ be a smooth solution of \eqref{eq:linear parabolic} satisfying \eqref{eq:ellipticity}. Assume $|\nabla u| \leq 1$ everywhere and
\[ \left\vert \left\{ (x,t) \in Q_1 : |\nabla u - e| > \eps_0 \right\} \right\vert \leq \eps_1\]
for some $e\in\mathbb S^{n-1}$ and two positive constants $\eps_0,\eps_1.$ Then, if $\eps_0$ and $\eps_1$ are sufficiently small, there exists a constant $a \in \R$, such that
\[ |u(x,t) - a - e \cdot x| \leq \eta \quad \text{ for all } (x,t) \in Q_{1/2}.\]
Here, both $\va_0$ and $\va_2$ depend only on $n,\lambda,\Lambda$ and $\eta$.
\end{lem}

\begin{proof}
Let $f(t):=| \{ x \in B_1 : |\nabla u(x,t) - e| > \eps_0 \} |$. By the assumptions and Fubini's theorem, we have that $\int_{-1}^0 f(t) \ud t\leq \eps_1$. It follows that for $E:=\{t\in (-1,0): f(t)\ge\sqrt{\eps_1}\}$, we obtain
\[
|E|\le \frac{1}{\sqrt{\eps_1}}\int _{E} f(t)\ud t\le \frac{1}{\sqrt{\eps_1}}\int _{-1}^0 f(t)\ud t\le \sqrt{\eps_1}.
\]
Therefore, for all $t \in (-1,0] \setminus E$, with $|E| \leq  \sqrt{\eps_1}$, we have
\begin{equation}\label{eq:aux-1}
 | \{ x \in B_1 : |\nabla u(x,t) - e| > \eps_0 \} | \leq \sqrt{\eps_1}.
 \end{equation}
It follows from \eqref{eq:aux-1} and Morrey's inequality (see, e.g., Section 5.6.2 in the book \cite{Evans}) that for all $t\in  (-1,0] \setminus E$, we have
\begin{equation}\label{eq:aux-2}
\osc_{B_{1/2}} (u(\cdot,t)-e\cdot x)\le C(n)\|\nabla u-e\|_{L^{2n}(B_1)}\le C(n)(\eps_0+\eps_1^{\frac{1}{4n}}),
\end{equation}
where $C(n)>0$ depends only on $n$. 

Meanwhile, since $|\nabla u| \leq 1$ in $Q_1$, we have that $\osc_{B_1} u(\cdot,t) \leq 2$ for all $t \in (-1,0]$. Thus, applying Lemma \ref{l:timeslices2allcylinder}, we have that $\osc_{Q_1} u \leq C$ for some constant $C$. The function $u$ is a solution of a uniformly parabolic equation. By Theorem \ref{thm:interior holder}, we have
\[ \|u\|_{C^\alpha(Q_{1/2})} \leq C\]
for some positive constants $\alpha$ and $C$ depending only on $\lambda,\Lambda, n$. Therefore, by \eqref{eq:aux-2} and the fact that $|E|\le\sqrt{\eps_1}$, we obtain 
\[
\osc_{B_{1/2}} (u(\cdot,t)-e\cdot x)\le C(\eps_0+\eps_1^{\frac{1}{4n}}+\eps_1^{\frac{\alpha}{4}})
\]
for all $t \in (-1/4,0]$ (that is, including $t \in E$). By Lemma \ref{l:timeslices2allcylinder}, we obtain
\[
\osc_{Q_{1/2}}  (u-e\cdot x)\le C(\eps_0+\eps_1^{\frac{1}{4n}}+\eps_1^{\frac{\alpha}{4}}),
\]
where $C>0$ depends only on $\lambda,\Lambda, n$. Hence, if $\eps_0$ and $\eps_1$ are sufficiently small, there exists a constant $a \in \R$, such that
\[ |u(t,x) - a - e \cdot x| \leq \eta \quad\text{ for all } (x,t) \in Q_{1/2}.\]
\end{proof}

\subsection{Iteration}
In this section, we finish our proof of the following a priori estimates.
\begin{thm}\label{thm:holdergradient}
Let $u$ be a smooth solution of \eqref{eq:main va} in $Q_1$. Then there exist two positive constants $\alpha$ and $C$ depending only on $n$ and $p$ such that 
\[
 \|\nabla u\|_{C^\alpha(Q_{1/2})}\le C(\|u\|_{L^\infty(Q_1)}+\va).
\]
Also, there holds
\[
\sup_{(x,t), (x,s)\in Q_{1/2}}\frac{|u(x,t)-u(x,s)|}{|t-s|^{\frac{1+\alpha}{2}}}\le C(\|u\|_{L^\infty(Q_1)}+\va).
\]
\end{thm}
\begin{proof}
We first show the H\"older estimate of $\nabla u$ at $(0,0)$. Moreover, by normalization, we may assume that $u(0,0)=0$ and $|\nabla u|\le 1$ in $Q_1$.

Let $\eta$ be the one in Theorem \ref{l:smallness to regularity} with $\gamma=1/2$, and for this $\eta$, let $\va_0,\va_1$ be two sufficiently small positive constants so that the conclusion of Lemma \ref{l:gradientarounde2smalloscillation} holds. For $\ell=1-\eps_0^2/2$ and $\mu=\va_1/|Q_1|$, if 
\[
|\{(x,t)\in Q_1: \nabla u\cdot e \le \ell\}|\le \mu |Q_1|\quad\mbox{for any }e\in\mathbb S^{n-1},
\]
then
\[ \left\vert \left\{ (x,t) \in Q_1 : |\nabla u - e| > \eps_0 \right\} \right\vert \leq \eps_1.\]
This is because if $|\nabla u(x,t)-e|>\eps_0$ for some $(x,t)\in Q_1$, then
\[
|\nabla u|^2-2\nabla u\cdot e+1\ge\eps_0^2.
\]
Since $|\nabla u|\le 1$, we have 
\[
\nabla u\cdot e\le 1-\eps_0^2/2.
\]
Therefore, if $\ell=1-\eps_0^2/2$ and $\mu=\va_1/|Q_1|$, then
\[
\left\{ (x,t) \in Q_1 : |\nabla u - e| > \eps_0 \right\}\subset  \{(x,t)\in Q_1: \nabla u\cdot e \le \ell\},
\]
from which it follows that
\[
\left\vert\left\{ (x,t) \in Q_1 : |\nabla u - e| > \eps_0 \right\}\right\vert \le  \left\vert\{(x,t)\in Q_1: \nabla u\cdot e\le \ell\}\right\vert\le \mu |Q_1|\le \eps_1.
\]

Let $\tau,\delta$ be the constants in Corollary \ref{cor:osc}. Let $k$ be the minimum nonnegative integer such that the condition \eqref{eq:stopping condition} does not hold. If $k=\infty$, then it follows immediately from Corollary \ref{cor:osc} that
\[
|\nabla u(x,t)|\le C(|x|+\sqrt{|t|})^\alpha\quad\mbox{for all }(x,t)\in Q_1,
\]
where $C=(1-\delta)^{-1}$ and $\alpha=\log(1-\delta)/\log\tau$.% are two positive constants depending only on $n$ and $p$.

If $k$ is finite, then
\[
|\{(x,t)\in Q_{\tau^k}: \nabla u\cdot e\le \ell (1-\delta)^k\}|\le \mu |Q_{\tau^k}|\quad\mbox{for some }e\in\mathbb S^{n-1}.
\]
Let \[v(x,t)=\frac{1}{\tau^k (1-\delta)^k}u(\tau^k x, \tau^{2k} t).\] Then $v$ satisfies
\[
v_t=\Delta v+(p-2)\frac{v_iv_j}{|\nabla v|^2+\eps (1-\delta)^{-2k}} v_{ij}\quad\mbox{in } Q_1,
\]
and
\[
|\{(x,t)\in Q_{1}: \nabla v\cdot e\le \ell\}|\le \mu |Q_{1}|\quad\mbox{for some }e\in\mathbb S^{n-1}.
\]
Consequently,
\[ \left\vert \left\{ (x,t) \in Q_1 : |\nabla v - e| > \eps_0 \right\} \right\vert \leq \eps_1.\]
Since condition \eqref{eq:stopping condition} holds for $k-1$, then $|\nabla v|\le 1$ in $Q_1$. It follows from Lemma \ref{l:gradientarounde2smalloscillation} that there exists $a\in\R$ such that
\[ |v(x,t) - a - e \cdot x| \leq \eta \quad \text{ for all } (x,t) \in Q_{1/2}.\]
By Theorem \ref{l:smallness to regularity} that 
there exists $b\in\R^n$ such that
\[
|\nabla v-b|\le C (|x|+\sqrt{|t|})  \quad\mbox{for all }(x,t)\in Q_{1/4}.
\]
Since $|\nabla v|\le 1$ and $|b|\le 1$,  there also holds
\[
|\nabla v-b|\le C(|x|+\sqrt{|t|})  \quad\mbox{for all }(x,t)\in Q_1.
\]
Rescaling back, we obtain
\[
|\nabla u-(1-\delta)^k b|\le C(1-\delta)^k \tau^{-k}(|x|+\sqrt{|t|}) \le C(|x|+\sqrt{|t|})^\alpha\quad\mbox{for all }(x,t)\in Q_{\tau^k}.
\]

On the other hand, we know that
\[
|\nabla u|< (1-\delta)^{i}\quad\mbox{in }Q_{\tau^{i}} \mbox{ and for all }i=0,\cdots, k.
\]
This implies that
\[
|\nabla u-(1-\delta)^k b|\le C(|x|+\sqrt{|t|})^\alpha \quad\mbox{for all }(x,t)\in Q_{1}\setminus Q_{\tau^k}.
\]
Therefore,
\[
|\nabla u-(1-\delta)^k b|\le C(|x|+\sqrt{|t|})^\alpha \quad\mbox{for all }(x,t)\in Q_{1}.
\]

In conclusion, we have proved that there exist $q\in\R^n$ with $|q|\le 1$, and two positive constants $\alpha, C$ such that
\[
|\nabla u(x,t)-q|\le C(|x|+\sqrt{|t|})^\alpha \quad\mbox{for all }(x,t)\in Q_{1}.
\]
By standard translation arguments, it follows that
\begin{equation}\label{eq:appr c1alpha}
 \|\nabla u\|_{C^\alpha(Q_{1/2})}\le C(\|u\|_{L^\infty(Q_1)}+\va).
\end{equation}

Now, we are going to prove the $C^{\frac{1+\alpha}{2}}$ continuity of $u$ in the time variable $t$.

Let $t\in[-1/4,0)$ and $r=\sqrt{|t|}$. For $(y,s)\in Q_r$, let
\[
w(y,s)=u(y,s)-u(0,0)-\nabla u(0,0)\cdot y.
\]
By \eqref{eq:appr c1alpha}, we have
\begin{equation}\label{eq:appr c1alpha2}
|u(y,s)-u(0,s)-\nabla u(0,s)\cdot y|\le C(\|u\|_{L^\infty(Q_1)}+\va) |y|^{1+\alpha},
\end{equation}
Therefore, for $y_1,y_2\in B_r$, 
\begin{align*}
&|w(y_1, s)-w(y_2,s)|\\
&=|u(y_1,s)-u(y_2,s)-\nabla u(0,0)\cdot (y_1-y_2)|\\
&\le |(\nabla u(0,s)-\nabla u(0,0))\cdot (y_1-y_2)|+C(\|u\|_{L^\infty(Q_1)}+\va) r^{1+\alpha}\\
&\le C(\|u\|_{L^\infty(Q_1)}+\va) |s|^{\frac{\alpha}{2}}|y_1-y_2|+C(\|u\|_{L^\infty(Q_1)}+\va) r^{1+\alpha}\\
&\le C(\|u\|_{L^\infty(Q_1)}+\va) r^{1+\alpha},
\end{align*}
where in the first inequality we used \eqref{eq:appr c1alpha2} and in the second inequality we used \eqref{eq:appr c1alpha}. Since $u$ satisfies \eqref{eq:main va}, $w$ satisfies a uniformly parabolic equation as well. By Lemma \ref{l:timeslices2allcylinder}, we have
\[
\osc_{Q_r} w\le C(\|u\|_{L^\infty(Q_1)}+\va) r^{1+\alpha}.
\]
In particular,
\[
|u(0,t)-u(0,0)|\le C(\|u\|_{L^\infty(Q_1)}+\va) |t|^{\frac{1+\alpha}{2}}.
\]
By standard translation arguments, it follows that
\[
\sup_{(x,t), (x,s)\in Q_{1/2}}\frac{|u(x,t)-u(x,s)|}{|t-s|^{\frac{1+\alpha}{2}}}\le C(\|u\|_{L^\infty(Q_1)}+\va).
\]

This finishes the proof of this theorem.
\end{proof}

\section{Approximations and the proof of our main result}\label{sec:appr}
This section is devoted to the final step of our proof of Theorem \ref{thm:mainholdergradient}, that is the approximation step.

Note that \eqref{eq:main va} is a uniformly parabolic quasilinear equation and its coefficients $a_{ij}(q)$ as in \eqref{eq:aij} are smooth with bounded derivatives (for each value of $\eps>0$). The next lemma follows directly from classical quasilinear equations theory (see, e.g., Theorem 4.4 of \cite{LSU} in page 560) and the Schauder estimates.

\begin{lem}\label{lem:app1}
Let $g\in C(\partial_p Q_1)$. For $\va>0$, there exists a unique solution $u^\va\in C^{\infty}(Q_1)\cap C(\overline Q_1)$ of \eqref{eq:main va} such that $u^\va=g$ on $\partial_p Q_1$.
\end{lem}

%Let $g_k\in C^\infty(\overline\Omega)$ be a sequence of functions such that $g_k\to g$  in $C(\overline Q_1)$ as $k\to\infty$. It follows from classical quasilinear equations theory (see, e.g., Theorem 4.4 of \cite{LSU} in page 560) and Schauder estimates that there exists a unique solution  $u_k^\va\in C^\infty(Q_1)\cap C(\overline Q_1)$  of  \eqref{eq:main va} such that $u^\va_k=g_k$ on $\partial_p Q_1$. By the maximum principle, we have $\|u^\va_k\|_{L^\infty(Q_1)}\le \|g_k\|_{L^\infty(\partial_p Q_1)}$. Moreover, from Theorem \ref{thm:holdergradient}, Schauder estimates and bootstrap arguments, we have that for every $l,m\in\mathbb N$ and $r\in (0,1)$,
%\[\|\nabla^l_t \nabla^m_x \, u^\va_k\|_{L^\infty(Q_{r})}\le C(r,l,m,n,p,\va, \sup_{k}\|g_k\|_{L^\infty(\partial_p Q_1)}).\]

%By comparison principles, we have that $\|u_k^\va-u_j^\va\|_{L^\infty(Q_1)}\le \|g_k-g_j\|_{L^\infty(Q_1)}$. Therefore, subject to passing to a subsequence, $u_k^\va$ converges to a function $u^\va\in C^\infty(Q_1)\cap C(\overline Q_1)$, which satisfies \eqref{eq:main va} in $Q_1$ and  $u^\va=g$ on $\partial_p Q_1$. The uniqueness of such solutions follows again  from comparison principles.

We are now ready to prove Theorem \ref{thm:mainholdergradient} by taking $\eps \to 0$ in the a priori estimate of Theorem \ref{thm:holdergradient}.

\begin{proof}[Proof of Theorem \ref{thm:mainholdergradient}]
Without loss of generality, we assume that $u\in C(\overline Q_1)$. Let $\omega$ be its modulus of continuity in $\overline Q_1$. By Lemma \ref{lem:app1}, for $\va\in (0,1)$, there exists a unique solution $v^\va \in C^\infty(Q_1)\cap C(\overline Q_1)$ of \eqref{eq:main va} such that $v^\va=u$ on $\partial_p Q_1$. Moreover, it follows from Theorem \ref{prop:boundary regularity} that there exists a modulus of continuity $\omega^*$, which depends only on $n,p,\omega, \|u\|_{L^\infty(\partial_p Q_1)}$, such that
\[
 |v^\va(x,t)-v^\va(y,s)|\le\omega^*(|x-y|\vee\sqrt{|s-t|})\quad\mbox{ for all }(x,t), (y,s)\in\overline Q_1.
\]
By the maximum principle, 
\[
 \|v^\va\|_{L^\infty(Q_1)} \le \|u\|_{L^\infty(\partial_p Q_1)}.
\]
It follows from Ascoli-Arzela theorem that there exists a subsequence $\{v^{\va_k}\}$ such that $v^{\va_k}\to v\in C(\overline Q_1)$ uniformly in $\overline Q_1$ as $\va_k\to 0$. By the stability property in Theorem \ref{thm:stability}, $v$ is a viscosity solution of \eqref{eq:main1}. By the comparison principle in Theorem \ref{thm:comparison principle}, we obtain that $u\equiv v$ in $\overline Q_1$.

On the other hand, it follows from Theorem \ref{thm:holdergradient} that, subject to a subsequence, $\nabla v^{\va_k}$ converges in $C^\al(Q_{1/2})$ for some constant $\al$ depending only on $n$ and $p$. Therefore, $u$ is differentiable in $x$ everywhere in $Q_{1/2}$, and thus, $\nabla v^{\va_k}$ converges to $\nabla u$ in $C^\al(Q_{1/2})$. Since 
\[
 \|\nabla v^{\va_k}\|_{C^\alpha(Q_{1/2})}\le C(\|v^{\va_k}\|_{L^\infty(Q_1)}+\va_k)\le C(\|u\|_{L^\infty(Q_1)}+\va_k),
\]
where $C>0$ depends only on $n$ and $p$, we obtain
\[
 \|\nabla u\|_{C^\alpha(Q_{1/2})}\le C\|u\|_{L^\infty(Q_1)}
\]
by sending $k\to\infty$.

We also know from Theorem \ref{thm:holdergradient} that for all $(x,t), (x,s)\in Q_{1/2}$, there holds
\[
|v^{\va_k}(x,t)-v^{\va_k}(x,s)|\le C(\|u\|_{L^\infty(Q_1)}+\va_k) |t-s|^{\frac{1+\alpha}{2}}.
\]
By sending $k\to\infty$, we obtain
\[
\sup_{(x,t), (x,s)\in Q_{1/2}}\frac{|u(x,t)-u(x,s)|}{|t-s|^{\frac{1+\alpha}{2}}}\le C\|u\|_{L^\infty(Q_1)}.
\]

This finishes the proof of Theorem \ref{thm:mainholdergradient}.
\end{proof}

\appendix
\section{Appendix A} \label{A}

In this section we provide a proof of Lemma \ref{lem:subsolution}

\begin{proof} [Proof of Lemma \ref{lem:subsolution}.]
In the following, we  denote $V=|\nabla u|^2+\va^2$. First,
\begin{align*}
\partial_t \varphi&= pV^\frac{p-2}{2} \nabla u\cdot \nabla u_t\\
&= pV^\frac{p-2}{2} u_k \Big(\Delta u_k-2(p-2)V^{-2}u_l u_{kl} u_i u_j u_{ij}+2(p-2)V^{-1}u_{ik} u_{ij}u_j\\
&\quad\quad+(p-2) V^{-1}u_iu_ju_{ijk}\Big)\\
&=p V^\frac{p-2}{2} \Big(u_k\Delta u_k-2(p-2)V^{-2}(\Delta_\infty u)^2+2(p-2)V^{-1}|\nabla^2 u  \nabla u|^2\\
&\quad\quad+(p-2) V^{-1}u_iu_ju_ku_{ijk}\Big),
\end{align*}
where $\Delta_\infty u=\sum_{i,j} u_{ij}u_iu_j$. Secondly,
\[
\partial_j \varphi=  pV^\frac{p-2}{2} u_k  u_{kj},
\]
and therefore,
\begin{align*}
\partial_{ij} \varphi &= p(p-2)V^\frac{p-4}{2}u_l u_{li} u_k  u_{kj}+ pV^\frac{p-2}{2}u_{ki} u_{kj}+ pV^\frac{p-2}{2}u_{k} u_{kij}.
\end{align*}
Consequently,
\begin{align*}
a_{ij}(\nabla u) \partial_{ij} \varphi&=p(p-2)V^\frac{p-4}{2}u_l u_{li} u_k  u_{ki}+ pV^\frac{p-2}{2}u_{ki} u_{ki}+ pV^\frac{p-2}{2}u_{k} u_{kii}\\
&\quad + p(p-2)^2 V^\frac{p-6}{2} u_i u_j u_l u_{li} u_k  u_{kj} \\
&\quad + p(p-2)V^\frac{p-4}{2} u_{ki} u_{kj}u_iu_j\\
&\quad +p(p-2)V^\frac{p-4}{2} u_i u_j u_k  u_{ijk} \\
&=2p(p-2)V^\frac{p-4}{2}|\nabla^2 u  \nabla u|^2 + pV^\frac{p-2}{2}|\nabla^2 u|^2 + pV^\frac{p-2}{2}u_{k}  \Delta u_k\\
&\quad + p(p-2)^2 V^\frac{p-6}{2}(\Delta_\infty u)^2\\
&\quad +p(p-2)V^\frac{p-4}{2} u_i u_j u_k  u_{ijk}.
\end{align*}
Therefore,
\begin{align*}
\big(\partial_t -a_{ij}(\nabla u)\partial_{ij}\big)\varphi=pV^\frac{p-6}{2}\Big(p(2-p)(\Delta_\infty u)^2-|\nabla^2 u|^2 V^2\Big)\le 0,
\end{align*}
where in the last inequality we used the H\"older inequality that
\begin{align*}
(\Delta_\infty u)^2&=\left(\sum_{i,j} u_{ij}u_iu_j\right)^2\\
&\le \left(\sum_{i,j} u_{ij}^2\right)\left(\sum_{i,j} u_i^2 u_j^2\right)=|\nabla^2 u|^2|\nabla u|^4\le |\nabla^2 u|^2 V^2.
\end{align*}
\end{proof}

\section{Appendix B}\label{sec:appB}
In this second appendix, we shall prove the boundary estimates in Proposition \ref{prop:boundary regularity}. Recall that for two real numbers $a$ and $b$, we denote $a\vee b=\max(a,b)$, $a\wedge b=\min(a,b)$.
\begin{lem}\label{lem:super1}
There exists a non negative continuous function $\psi: \R^n\times(-\infty,0]\to \R$ such that
\begin{itemize}
 \item $\psi=0\quad\mbox{in } B_1\times\{t=0\}$;
\item $\psi_t-a_{ij}(x,t)\psi_{ij}\ge 0 \quad\mbox{in } (\R^n\setminus B_1)\times (-\infty,0]$;
\item $\psi\ge 1 \quad\mbox{in } (\R^n\times (-\infty,0])\setminus (B_2\times [-1,0])$,
\end{itemize}
where $a_{ij}(x,t)$ satisfies \eqref{eq:ellipticity}.
\end{lem}
\begin{proof}
 Let $v(x)=\sqrt{(|x|-1)^+}$. It follows from elementary calculations that there exists $\delta\in(0,1)$ such that
\[
 -a_{ij}v_{ij}\ge 1\quad\mbox{for } 1<|x|<1+\delta.
\]
Then $\psi=\min(\delta^{-1/2}v(x)-t,1)$ is a desired function.
\end{proof}

\begin{lem}\label{lem: boundary1}
 Let $u\in C(\overline Q_1)$ be a solution of \eqref{eq:linear parabolic} satisfying \eqref{eq:ellipticity}. 
Let $(x,t)\in \partial B_1\times (-1,0]$ be fixed, $\rho$ be a modulus of continuity such that
\[
 |u(y,s)-u(x,t)|\le \rho(|x-y|\vee\sqrt{|t-s|})
\]
for all $(y,s)\in \partial_p(B_1\times (-1,t])$. Then there exists another modulus of continuity $\rho^*$ depending only on $n,\lambda,\Lambda, \rho, \|u\|_{L^\infty(\partial_p Q_1)}$ such that
\[
 |u(x,t)-u(y,s)|\le \rho^*(|x-y|\vee\sqrt{|t-s|})
\]
for all $(y,s)\in \overline B_1\times [-1,t]$.
\end{lem}
\begin{proof}
Fix $r\in (0,1)$.  
Let $x_r=(1+r)x$ and $\psi$ be as in Lemma \ref{lem:super1}. Define
\[
 v(y,s)=u(x,t)+\rho(3r)+2\|u\|_{L^\infty(\partial_p Q_1)}\psi\left(\frac{y-x_r}{r},\frac{s-t}{r^2}\right).
\]
Then
\[
v_s-a_{ij}v_{ij}\ge 0\quad\mbox{in } \Omega:=(B_{3r}(x)\cap B_1)\times(-1,t]. 
\]
For $(y,s)\in\partial_p \Omega$ and $|y-x|\vee \sqrt{|s-t|}< 3r$, then
\[
 v(y,s)\ge u(x,t)+\rho(3r)\ge u(y,s).
\]
For $(y,s)\in\partial_p \Omega$ and $|y-x|\vee \sqrt{|s-t|}\ge 3r$, then
\[
 v(y,s)\ge u(x,t)+2\|u\|_{L^\infty(\partial_p Q_1)}=u(x,t)+2\|u\|_{L^\infty(Q_1)}\ge u(y,s).
\]
It follows from the maximum principle that $v\ge u$ in $\Omega$, i.e., 
\[
\rho(3r)+2\|u\|_{L^\infty(\partial_p Q_1)}\psi\left(\frac{y-x_r}{r},\frac{s-t}{r^2}\right)\ge u(y,s)-u(x,t).
\]
Similarly, one can show that
\[
\rho(3r)+2\|u\|_{L^\infty(\partial_p Q_1)}\psi\left(\frac{y-x_r}{r},\frac{s-t}{r^2}\right)\ge u(x,t)-u(y,s).
\]
Therefore, for $(y,s)\in\overline \Omega$.
\begin{equation}\label{eq:boundary-aux}
 |u(x,t)-u(y,s)|\le \rho(3r)+2\|u\|_{L^\infty(\partial_p Q_1)}\psi\left(\frac{y-x_r}{r},\frac{s-t}{r^2}\right).
\end{equation}
It is clear from the definition of $\psi$ that \eqref{eq:boundary-aux} holds for $(y,s)\in (B_1\setminus B_{3r}(x))\times(-1,t]$ as well.  
Meanwhile
\[
\psi\left(\frac{y-x_r}{r},\frac{s-t}{r^2}\right)=\psi\left(\frac{y-x_r}{r},\frac{s-t}{r^2}\right)-\psi\left(\frac{x-x_r}{r},0\right)\le \overline\rho((|x-y|\vee \sqrt{|t-s|})/r),
\]
where $\overline \rho$ is a modulus continuity of $\psi$. Therefore, we have for $(y,s)\in\overline B_1\times [-1,t]$,
\[
  |u(x,t)-u(y,s)|\le \rho(3r)+2\|u\|_{L^\infty(\partial_p Q_1)}\overline\rho((|x-y|\vee \sqrt{|t-s|})/r).
\]
The conclusion then follows from the observation that
\[
 \rho^*(d)=\inf_{r\in (0,1)} \left(\rho(3r)+2\|u\|_{L^\infty(\partial_p Q_1)}\overline\rho(d/r)\right)
\]
is a modulus of continuity.
\end{proof}

\begin{lem}\label{lem: boundary2}
 Let $t\in [-1,0)$ and $u\in C(\overline B_1\times [t,0])$ be a solution of \eqref{eq:linear parabolic} in $B_1\times (t,0]$ satisfying \eqref{eq:ellipticity}. 
Let $x\in \overline B_1$ be fixed, $\rho$ be a modulus of continuity such that
\[
 |u(y,s)-u(x,t)|\le \rho(|x-y|\vee\sqrt{|s-t|})
\]
for all $(y,s)\in \partial_p ( B_1\times (t,0])$.  Then there exists another modulus of continuity $\rho^*$ depending only on $n,\lambda,\Lambda, \rho$ and  $\|u\|_{L^\infty(\partial_p (B_1\times (t,0]))}$ such that
\[
 |u(x,t)-u(y,s)|\le \rho^*(|x-y|\vee\sqrt{|s-t|})
\]
for all  $(y,s)\in \overline B_1\times [t,0]$.
\end{lem}
\begin{proof}
Let $b\in C^\infty(\R^n)$ be a nonnegative function such that $b\equiv 1$ in $\R^n\setminus B_1$ and $b(0)=0$. Let
\[
 \phi(y,s)=b(y)+Ms,
\]
where $M=\sup_{B_1\times (t,0]} |a_{ij}| \sup_{\R^n} |\nabla^2 b|+1$, and $\overline\rho$ be its modulus of continuity.  Define
\[
 v(y,s)=u(x,t)+\rho(r)+2\|u\|_{L^\infty(\partial_p (B_1\times (t,0]))}\phi\Big(\frac{y-x}{r},\frac{s-t}{r^2}\Big).
\]
Then
\[
v_s-a_{ij}v_{ij}\ge 0\quad\mbox{in } B_1\times (t,0]. 
\]
For $(y,s)\in\partial_p (B_1\times (t,0])$ and $|y-x|\vee \sqrt{|s-t|}< r$, then
\[
 v(y,s)\ge u(x,t)+\rho(r)\ge u(y,s).
\]
For $(y,s)\in\partial_p (B_1\times (t,0])$ and $|y-x|\vee \sqrt{|s-t|}\ge r$, then either $|y-x|\ge r$ or $|s-t|\ge r^2$, each of which implies that
\[
 v(y,s)\ge u(x,t)+2\|u\|_{L^\infty(\partial_p (B_1\times (t,0]))}=u(x,t)+2\|u\|_{L^\infty(B_1\times (t,0])}\ge u(y,s).
\]
It follows from the maximum principle that $v\ge u$ in $\overline Q_1$, i.e., 
\[
\rho(r)+2\|u\|_{L^\infty(\partial_p (B_1\times (t,0]))}\phi\Big(\frac{y-x}{r},\frac{s-t}{r^2}\Big)\ge u(y,s)-u(x,t).
\]
Similarly, one can show that
\[
\rho(r)+2\|u\|_{L^\infty(\partial_p (B_1\times (t,0]))}\phi\Big(\frac{y-x}{r},\frac{s-t}{r^2}\Big)\ge u(x,t)-u(y,s).
\]
Meanwhile
\[
\phi\left(\frac{y-x}{r},\frac{s-t}{r^2}\right)=\phi\left(\frac{y-x}{r},\frac{s-t}{r^2}\right)-\phi(0,0)\le \overline\rho((|x-y|\vee \sqrt{|s-t|})/r),
\]
where $\overline \rho$ is a modulus continuity of $\phi$. Therefore, we have
\[
  |u(x,t)-u(y,s)|\le \rho(r)+2\|u\|_{L^\infty(\partial_p (B_1\times (t,0]))}\overline\rho((|x-y|\vee \sqrt{|s-t|})/r).
\]
The conclusion then follows from the observation that
\[
 \rho^*(d)=\inf_{r\in (0,1)} \left(\rho(r)+2\|u\|_{L^\infty(\partial_p (B_1\times (t,0]))}\overline\rho(d/r)\right)
\]
is a modulus of continuity.
\end{proof}

\begin{lem}\label{lem: boundary3}
 Let $u\in C(\overline Q_1)$ be a solution of \eqref{eq:linear parabolic} satisfying \eqref{eq:ellipticity}. 
Let $(x,t)\in \partial B_1\times (-1,0]$ be fixed, $\rho$ be a modulus of continuity such that
\[
 |u(y,s)-u(x,t)|\le \rho(|x-y|\vee\sqrt{|t-s|})
\]
for all $(y,s)\in \partial_p Q_1$. Then there exists another modulus of continuity $\rho^*$ depending only on $n,\lambda,\Lambda, \rho, \|u\|_{L^\infty(\partial_p Q_1)}$ such that
\[
 |u(x,t)-u(y,s)|\le \rho^*(|x-y|\vee\sqrt{|t-s|})
\]
for all $(y,s)\in \overline Q_1$.
\end{lem}

\begin{proof}
It follows from Lemma \ref{lem: boundary1} that there exists a  modulus of continuity $\rho_1$ depending only on $n,\lambda,\Lambda, \rho, \|u\|_{L^\infty(\partial_p Q_1)}$ such that
\[
 |u(x,t)-u(y,s)|\le \rho_1(|x-y|\vee\sqrt{|t-s|})
\]
for all $(y,s)\in \overline B_1\times [-1,t]$. If  $t< 0$, by applying Lemma \ref{lem: boundary2} to the cylinder $B_1\times(t,0)$ and noticing that $\|u\|_{L^\infty(\partial_p (B_1\times (t,0]))}\le \|u\|_{L^\infty(Q_1)}\le \|u\|_{L^\infty(\partial_p Q_1)}$, we conclude that there exists a  modulus of continuity $\rho_2$ depending only on $n,\lambda,\Lambda, \rho, \|u\|_{L^\infty(\partial_p Q_1)}$ such that
\[
 |u(x,t)-u(y,s)|\le \rho_2(|x-y|\vee\sqrt{|t-s|})
\]
for all $(y,s)\in \overline B_1\times [t,0]$. Finally, the choice of $\rho^*=\rho_1+\rho_2$ is the desired one.
\end{proof}

\begin{cor}\label{cor: boundary}
 Let $u\in C(\overline Q_1)$ be a solution of \eqref{eq:linear parabolic} satisfying \eqref{eq:ellipticity}. 
Let $(x,t)\in \partial_p Q_1$ be fixed, $\rho$ be a modulus of continuity such that
\[
 |u(y,s)-u(x,t)|\le \rho(|x-y|\vee\sqrt{|t-s|})
\]
for all $(y,s)\in \partial_p Q_1$. Then there exists another modulus of continuity $\tilde\rho$ depending only on $n,\lambda,\Lambda, \rho, \|u\|_{L^\infty(\partial_p Q_1)}$ such that
\[
 |u(x,t)-u(y,s)|\le \tilde\rho(|x-y|\vee\sqrt{|t-s|})
\]
for all $(y,s)\in \overline Q_1$.
\end{cor}
\begin{proof}
 It follows from Lemma \ref{lem: boundary2} and Lemma \ref{lem: boundary3}.
\end{proof}

\begin{proof}[Proof of Proposition \ref{prop:boundary regularity}]
Let $(x,t), (y,s)\in Q_1$, and we assume that $t\ge s$. Let $$d_X=\min(1-|x|, \sqrt{t+1}),$$ and $x_0$ be such that $|x-x_0|=1-|x|$. Let $\tilde \rho$ be the one in the conclusion of Corollary \ref{cor: boundary}.

\medskip

\emph{Case 1:} $(1-|x|)^2\le (1+t)$. Then $d_X=1-|x|$. 

\medskip

If $(y,s)\in B_{d_X/2}(x)\times (t-d_X^2/4,t]$, then by the interior H\"older estimates Theorem \ref{thm:interior holder}, we have
\[
 d_X^\alpha\frac{|u(x,t)-u(y,s)|}{(|x-y|\vee\sqrt{t-s})^\alpha}\le C \|u-u(x_0,t)\|_{L^\infty(B_{d_X}(x)\times (t-d_X^2,t])}\le C\tilde\rho(2d_X).
\]
Suppose that $2^{-m-1}d_X\le|x-y|\vee\sqrt{t-s}\le 2^{-m}d_X$ for some integer $m\ge 1$. Then
\[
 |u(x,t)-u(y,s)|\le C \frac{\tilde \rho(2^{m+2}(|x-y|\vee\sqrt{t-s}))}{2^{m\alpha}}.
\]
Notice that
\[
 \rho_1(d):=C \sup_{m\ge 1}\frac{\tilde\rho(2^{m+2}d)}{2^{m\alpha}}
\]
is a modulus of continuity, and therefore,
\[
  |u(x,t)-u(y,s)|\le  \rho_1(|x-y|\vee\sqrt{t-s}).
\]

If $(y,s)\not\in B_{d_X/2}(x)\times (t-d_X^2/4,t]$, then
\[
 \begin{split}
  |u(x,t)-u(y,s)|&\le |u(x,t)-u(x_0,t)|+|u(x_0,t)-u(y,s)|\\
&\le \tilde\rho(d_X)+\tilde\rho(|x_0-y|\vee \sqrt{|t-s|})\\
&\le \tilde\rho(2(|x-y|\vee\sqrt{|t-s|}))+\tilde\rho((|x-y|+d_X)\vee \sqrt{|t-s|})\\
&\le \tilde\rho(2(|x-y|\vee\sqrt{|t-s|}))+\tilde\rho(3(|x-y|\vee \sqrt{|t-s|}))\\
&\le 2\tilde\rho(3(|x-y|\vee \sqrt{|t-s|})).
 \end{split}
\]

\medskip

\emph{Case 2:} $(1-|x|)^2\ge (1+t)$. Then $d_X=t+1$.

\medskip

As before, if $(y,s)\in B_{d_X/2}(x)\times (t-d_X^2/4,t]$, then we have
\[
 d_X^\alpha\frac{|u(x,t)-u(y,s)|}{(|x-y|\vee\sqrt{t-s})^\alpha}\le C \|u-u(x,-1)\|_{L^\infty(B_{d_X}(x)\times (t-d_X^2,t])}\le C\tilde\rho(2d_X),
\]
and therefore, 
\[
  |u(x,t)-u(y,s)|\le  \rho_1(|x-y|\vee\sqrt{t-s}).
\]
If $(y,s)\not\in B_{d_X/2}(x)\times (t-d_X^2/4,t]$, then
\[
 \begin{split}
  |u(x,t)-u(y,s)|&\le |u(x,t)-u(x,-1)|+|u(x,-1)-u(y,s)|\\
&\le \tilde\rho(d_X)+\tilde\rho(|x-y|\vee \sqrt{|1+s|})\\
&\le \tilde\rho(2(|x-y|\vee\sqrt{|t-s|}))+\tilde\rho(|x-y|\vee \sqrt{|1+t|})\\
&\le \tilde\rho(2(|x-y|\vee\sqrt{|t-s|}))+\tilde\rho(3(|x-y|\vee \sqrt{|t-s|}))\\
&\le 2\tilde\rho(3(|x-y|\vee \sqrt{|t-s|})).
 \end{split}
\]

\end{proof}

\bibliographystyle{abbrv}
\bibliography{plaplacian}

\begin{thebibliography}{10}

\bibitem{BG}
A.~Banerjee and N.~Garofalo.
\newblock Gradient bounds and monotonicity of the energy for some nonlinear
  singular diffusion equations.
\newblock {\em Indiana Univ. Math. J.}, 62(2):699--736, 2013.

\bibitem{BG3}
A.~Banerjee and N.~Garofalo.
\newblock Modica type gradient estimates for an inhomogeneous variant of the
  normalized {$p$}-{L}aplacian evolution.
\newblock {\em Nonlinear Anal.}, 121:458--468, 2015.

\bibitem{BG2}
A.~Banerjee and N.~Garofalo.
\newblock On the {D}irichlet boundary value problem for the normalized
  {$p$}-{L}aplacian evolution.
\newblock {\em Commun. Pure Appl. Anal.}, 14(1):1--21, 2015.

\bibitem{BEJ}
E.~N. Barron, L.~C. Evans, and R.~Jensen.
\newblock The infinity {L}aplacian, {A}ronsson's equation and their
  generalizations.
\newblock {\em Trans. Amer. Math. Soc.}, 360(1):77--101, 2008.

\bibitem{CC}
L.~A. Caffarelli and X.~Cabr{{\'e}}.
\newblock {\em Fully nonlinear elliptic equations}, volume~43 of {\em American
  Mathematical Society Colloquium Publications}.
\newblock American Mathematical Society, Providence, RI, 1995.

\bibitem{CGG}
Y.~G. Chen, Y.~Giga, and S.~Goto.
\newblock Uniqueness and existence of viscosity solutions of generalized mean
  curvature flow equations.
\newblock {\em J. Differential Geom.}, 33(3):749--786, 1991.

\bibitem{CM15}
T.~H. Colding and W.~P. Minicozzi.
\newblock Differentiability of the arrival time.
\newblock {\em arXiv:1501.07899}, 2015.

\bibitem{DiBenedetto2}
E.~DiBenedetto.
\newblock {$C^{1+\alpha }$} local regularity of weak solutions of degenerate
  elliptic equations.
\newblock {\em Nonlinear Anal.}, 7(8):827--850, 1983.

\bibitem{DiBenedetto}
E.~DiBenedetto.
\newblock {\em Degenerate parabolic equations}.
\newblock Universitext. Springer-Verlag, New York, 1993.

\bibitem{DBF}
E.~DiBenedetto and A.~Friedman.
\newblock H{\"o}lder estimates for nonlinear degenerate parabolic systems.
\newblock {\em J. Reine Angew. Math.}, 357:1--22, 1985.

\bibitem{KD}
K.~Does.
\newblock An evolution equation involving the normalized {$p$}-{L}aplacian.
\newblock {\em Commun. Pure Appl. Anal.}, 10(1):361--396, 2011.

\bibitem{Evans2}
L.~C. Evans.
\newblock A new proof of local {$C^{1,\alpha }$} regularity for solutions of
  certain degenerate elliptic p.d.e.
\newblock {\em J. Differential Equations}, 45(3):356--373, 1982.

\bibitem{Evans}
L.~C. Evans.
\newblock {\em Partial differential equations}, volume~19 of {\em Graduate
  Studies in Mathematics}.
\newblock American Mathematical Society, Providence, RI, 1998.

\bibitem{ESavin}
L.~C. Evans and O.~Savin.
\newblock {$C^{1,\alpha}$} regularity for infinity harmonic functions in two
  dimensions.
\newblock {\em Calc. Var. Partial Differential Equations}, 32(3):325--347,
  2008.

\bibitem{ESmart}
L.~C. Evans and C.~K. Smart.
\newblock Everywhere differentiability of infinity harmonic functions.
\newblock {\em Calc. Var. Partial Differential Equations}, 42(1-2):289--299,
  2011.

\bibitem{ESS}
L.~C. Evans, H.~M. Soner, and P.~E. Souganidis.
\newblock Phase transitions and generalized motion by mean curvature.
\newblock {\em Comm. Pure Appl. Math.}, 45(9):1097--1123, 1992.

\bibitem{ES1}
L.~C. Evans and J.~Spruck.
\newblock Motion of level sets by mean curvature. {I}.
\newblock {\em J. Differential Geom.}, 33(3):635--681, 1991.

\bibitem{ES2}
L.~C. Evans and J.~Spruck.
\newblock Motion of level sets by mean curvature. {II}.
\newblock {\em Trans. Amer. Math. Soc.}, 330(1):321--332, 1992.

\bibitem{ES3}
L.~C. Evans and J.~Spruck.
\newblock Motion of level sets by mean curvature. {III}.
\newblock {\em J. Geom. Anal.}, 2(2):121--150, 1992.

\bibitem{ES4}
L.~C. Evans and J.~Spruck.
\newblock Motion of level sets by mean curvature. {IV}.
\newblock {\em J. Geom. Anal.}, 5(1):77--114, 1995.

\bibitem{FS}
E.~Ferretti and M.~V. Safonov.
\newblock Growth theorems and {H}arnack inequality for second order parabolic
  equations.
\newblock In {\em Harmonic analysis and boundary value problems
  ({F}ayetteville, {AR}, 2000)}, volume 277 of {\em Contemp. Math.}, pages
  87--112. Amer. Math. Soc., Providence, RI, 2001.

\bibitem{IJS}
C.~Imbert, T.~Jin, and L.~Silvestre.
\newblock H{\"o}lder gradient estimates for a class of singular or degenerate
  parabolic equations.
\newblock {\em In preparation}.

\bibitem{ImbertS}
C.~Imbert and L.~Silvestre.
\newblock An introduction to fully nonlinear parabolic equations.
\newblock In {\em An introduction to the {K}{\"a}hler-{R}icci flow}, volume
  2086 of {\em Lecture Notes in Math.}, pages 7--88. Springer, Cham, 2013.

\bibitem{IS}
H.~Ishii and P.~Souganidis.
\newblock Generalized motion of noncompact hypersurfaces with velocity having
  arbitrary growth on the curvature tensor.
\newblock {\em Tohoku Math. J. (2)}, 47(2):227--250, 1995.

\bibitem{Jensen}
R.~Jensen.
\newblock Uniqueness of {L}ipschitz extensions: minimizing the sup norm of the
  gradient.
\newblock {\em Arch. Rational Mech. Anal.}, 123(1):51--74, 1993.

\bibitem{Juutinen2014}
P.~Juutinen.
\newblock Decay estimates in the supremum norm for the solutions to a nonlinear
  evolution equation.
\newblock {\em Proc. Roy. Soc. Edinburgh Sect. A}, 144(3):557--566, 2014.

\bibitem{JK}
P.~Juutinen and B.~Kawohl.
\newblock On the evolution governed by the infinity {L}aplacian.
\newblock {\em Math. Ann.}, 335(4):819--851, 2006.

\bibitem{KKK}
B.~Kawohl, S.~Kr{{\"o}}mer, and J.~Kurtz.
\newblock Radial eigenfunctions for the game-theoretic {$p$}-{L}aplacian on a
  ball.
\newblock {\em Differential Integral Equations}, 27(7-8):659--670, 2014.

\bibitem{KSS}
R.~V. Kohn and S.~Serfaty.
\newblock A deterministic-control-based approach to motion by curvature.
\newblock {\em Comm. Pure Appl. Math.}, 59(3):344--407, 2006.

\bibitem{Krylov}
N.~V. Krylov.
\newblock Some {$L_p$}-estimates for elliptic and parabolic operators with
  measurable coefficients.
\newblock {\em Discrete Contin. Dyn. Syst. Ser. B}, 17(6):2073--2090, 2012.

\bibitem{KS}
N.~V. Krylov and M.~V. Safonov.
\newblock A property of the solutions of parabolic equations with measurable
  coefficients.
\newblock {\em Izv. Akad. Nauk SSSR Ser. Mat.}, 44(1):161--175, 239, 1980.

\bibitem{LSU}
O.~A. Lady{\v{z}}enskaja, V.~A. Solonnikov, and N.~N. Ural'ceva.
\newblock {\em Linear and quasilinear equations of parabolic type}.
\newblock Translated from the Russian by S. Smith. Translations of Mathematical
  Monographs, Vol. 23. American Mathematical Society, Providence, R.I., 1968.

\bibitem{LewMan}
M.~Lewicka and J.~J. Manfredi.
\newblock Game theoretical methods in {PDE}s.
\newblock {\em Boll. Unione Mat. Ital.}, 7(3):211--216, 2014.

\bibitem{Lewis}
J.~L. Lewis.
\newblock Regularity of the derivatives of solutions to certain degenerate
  elliptic equations.
\newblock {\em Indiana Univ. Math. J.}, 32(6):849--858, 1983.

\bibitem{Lin86}
F.-H. Lin.
\newblock Second derivative {$L^p$}-estimates for elliptic equations of
  nondivergent type.
\newblock {\em Proc. Amer. Math. Soc.}, 96(3):447--451, 1986.

\bibitem{LS2015}
Q.~Liu and A.~Schikorra.
\newblock General existence of solutions to dynamic programming equations.
\newblock {\em Commun. Pure Appl. Anal.}, 14(1):167--184, 2015.

\bibitem{MPR}
J.~J. Manfredi, M.~Parviainen, and J.~D. Rossi.
\newblock An asymptotic mean value characterization for a class of nonlinear
  parabolic equations related to tug-of-war games.
\newblock {\em SIAM J. Math. Anal.}, 42(5):2058--2081, 2010.

\bibitem{MPR2012}
J.~J. Manfredi, M.~Parviainen, and J.~D. Rossi.
\newblock Dynamic programming principle for tug-of-war games with noise.
\newblock {\em ESAIM Control Optim. Calc. Var.}, 18(1):81--90, 2012.

\bibitem{Oberman2013}
A.~M. Oberman.
\newblock Finite difference methods for the infinity {L}aplace and
  {$p$}-{L}aplace equations.
\newblock {\em J. Comput. Appl. Math.}, 254:65--80, 2013.

\bibitem{PSSW}
Y.~Peres, O.~Schramm, S.~Sheffield, and D.~B. Wilson.
\newblock Tug-of-war and the infinity {L}aplacian.
\newblock {\em J. Amer. Math. Soc.}, 22(1):167--210, 2009.

\bibitem{PS}
Y.~Peres and S.~Sheffield.
\newblock Tug-of-war with noise: a game-theoretic view of the
  {$p$}-{L}aplacian.
\newblock {\em Duke Math. J.}, 145(1):91--120, 2008.

\bibitem{Rossi2011}
J.~D. Rossi.
\newblock Tug-of-war games and {PDE}s.
\newblock {\em Proc. Roy. Soc. Edinburgh Sect. A}, 141(2):319--369, 2011.

\bibitem{RossiGAMES}
J.~D. Rossi.
\newblock Tug-of-war games and {PDE}s.
\newblock {\em Proc. Roy. Soc. Edinburgh Sect. A}, 141(2):319--369, 2011.

\bibitem{Rudd2015}
M.~Rudd.
\newblock Statistical exponential formulas for homogeneous diffusion.
\newblock {\em Commun. Pure Appl. Anal.}, 14(1):269--284, 2015.

\bibitem{Savin}
O.~Savin.
\newblock {$C^1$} regularity for infinity harmonic functions in two dimensions.
\newblock {\em Arch. Ration. Mech. Anal.}, 176(3):351--361, 2005.

\bibitem{Savin07}
O.~Savin.
\newblock Small perturbation solutions for elliptic equations.
\newblock {\em Comm. Partial Differential Equations}, 32(4-6):557--578, 2007.

\bibitem{Spencer}
J.~Spencer.
\newblock Balancing games.
\newblock {\em J. Combinatorial Theory Ser. B}, 23(1):68--74, 1977.

\bibitem{Tolksdorf}
P.~Tolksdorf.
\newblock Regularity for a more general class of quasilinear elliptic
  equations.
\newblock {\em J. Differential Equations}, 51(1):126--150, 1984.

\bibitem{Uhlenbeck}
K.~Uhlenbeck.
\newblock Regularity for a class of non-linear elliptic systems.
\newblock {\em Acta Math.}, 138(3-4):219--240, 1977.

\bibitem{Ural}
N.~N. Ural'ceva.
\newblock Degenerate quasilinear elliptic systems.
\newblock {\em Zap. Nau\v cn. Sem. Leningrad. Otdel. Mat. Inst. Steklov.
  (LOMI)}, 7:184--222, 1968.

\bibitem{lihewang}
L.~Wang.
\newblock Compactness methods for certain degenerate elliptic equations.
\newblock {\em J. Differential Equations}, 107(2):341--350, 1994.

\bibitem{wang}
Y.~Wang.
\newblock Small perturbation solutions for parabolic equations.
\newblock {\em Indiana Univ. Math. J.}, 62(2):671--697, 2013.

\bibitem{Wiegner}
M.~Wiegner.
\newblock On {$C_\alpha$}-regularity of the gradient of solutions of degenerate
  parabolic systems.
\newblock {\em Ann. Mat. Pura Appl. (4)}, 145:385--405, 1986.

\end{thebibliography}

%\index{Bibliography@\emph{Bibliography}}%

\bigskip

\bigskip

\noindent T. Jin

\noindent Department of Mathematics, The Hong Kong University of Science and Technology\\
Clear Water Bay, Kowloon, Hong Kong

\smallskip
and
\smallskip

\noindent Department of Computing and Mathematical Sciences, California Institute of Technology \\
1200 E. California Blvd., MS 305-16, Pasadena, CA 91125, USA\\[1mm]
Email: \textsf{tianlingjin@ust.hk} / \textsf{tianling@caltech.edu}

\bigskip

\noindent L. Silvestre

\noindent Department of Mathematics, The University of Chicago\\
5734 S. University Avenue, Chicago, IL 60637, USA\\[1mm]
Email: \textsf{luis@math.uchicago.edu}

\end{document}